\documentclass[12pt]{amsart}
\usepackage[utf8]{inputenc}
\usepackage{amsmath,amsfonts,amsthm,amssymb}
\usepackage{enumitem}
\usepackage{graphicx}
\usepackage{pgf,tikz,pgfplots}
\usetikzlibrary{arrows}
\usetikzlibrary{angles,quotes}
\usetikzlibrary{calc}
\usepgfplotslibrary{fillbetween}
\usetikzlibrary{positioning}
\usetikzlibrary{calc,shapes.geometric}
\usetikzlibrary{decorations.pathreplacing}
\usetikzlibrary{cd}
\usetikzlibrary{positioning}
\usepackage{mathrsfs}
\usepackage{caption,subcaption}
\usepackage{mathtools}
\usepackage{xspace}
\usetikzlibrary{patterns}
\calclayout
\usepackage{soul}
\usepackage{csquotes}
\usepackage{comment}
\usepackage{float}
\usepackage{amsmath}
\usepackage{rotating}
\usepackage[textsize=tiny]{todonotes}

\usepackage{hyperref}
\hypersetup{
  colorlinks   = true,          
  urlcolor     = blue,          
  linkcolor    = green,          
  citecolor   = orange             
}

\newtheorem{theorem}{Theorem}
\newtheorem{corollary}[theorem]{Corollary}

\newtheorem{lemma}[theorem]{Lemma}
\newtheorem{question}[theorem]{Question}

\theoremstyle{definition}
\newtheorem{definition}[theorem]{Definition}
\newtheorem{example}[theorem]{Example}

\theoremstyle{remark}
\newtheorem{remark}[theorem]{Remark}

\DeclareMathOperator{\conv}{Conv}

\DeclareMathOperator{\Gr}{\mathrm{Gr}}
\DeclareMathOperator{\GL}{\mathrm{GL}}
\DeclareMathOperator{\Trop}{\mathrm{Trop}}
\DeclareMathOperator{\TGr}{\mathrm{TropGr}}

\renewcommand{\P}{\mathrm{P}}
\newcommand{\M}{\mathsf{M}}
\newcommand{\U}{\mathsf{U}}
\newcommand{\Dr}{\mathrm{Dr}}
\newcommand{\LPM}{\mathrm{LPM}}

\newcommand\doi[1]{\href{http://dx.doi.org/#1}{\texttt{doi:#1}}}

\tikzstyle{lattice} = [draw=black, fill=black]
\tikzstyle{valid} = [draw=red, thick, fill=white]
\tikzstyle{intersect} = [draw=orange, fill=orange]
\tikzstyle{boundary} = [draw=blue, fill=blue]
\tikzstyle{triangle} = [draw=black, thick, fill=green!20]
\tikzstyle{inequality} = [draw=green, thick]
\tikzstyle{someline} = [draw=black, dashed]

\title[LPM Subdivisions, Positive Tropical Grassmannian and Amplituhedron]{Lattice Path Matroidal Subdivisions, Positive Tropical Grassmannian and Amplituhedron}
\author{Ayush Kumar Tewari and Ahmed Umer Ashraf}

\address[A.~K.~Tewari]{}
\email{tewari@math.tu-berlin.de}

\address[A.~U.~Ashraf]{Department of Mathematical and
Computational Sciences, University of Toronto,
Mississauga, Ontario, Canada.}
\email{aashra90@gmail.com}

\makeatletter
\@namedef{subjclassname@1991}{Mathematics Subject Classification (2020)}
\makeatother

\subjclass{52B40, 14T15, 81U99}

\keywords{Lattice path matroidal subdivision, LPMfan, LPM polytope decomposition}

\thanks{We would like to thank Michael Joswig and Luis Ferroni for going through earlier drafts of this article and for providing valuable suggestions and comments. We are also thankful to David Speyer for providing his comments and for pointing us to Luis Ferroni and his work on lattice path matroids.}
\pgfplotsset{compat=1.17}
\begin{document}

\maketitle

\begin{abstract}
We introduce the notion of lattice path matroidal subdivisions, or LPM subdivisions for short, and show that these subdivisions are regular and hence the weight vectors for them lie in the Dressian. This leads us to explore the structure of the set of these weights inside the Dressian and owing to the fact that Lattice path matroids are positroids, we move to the positive Dressian which in turn is equal to the positive tropical Grassmannian, an object of immense interest currently in Physics. This is related to the amplituhedron and  positive configuration space, which we describe here and wish to explore these connections further.
\end{abstract}

{
\hypersetup{linkcolor=blue}
\tableofcontents
}

\section{Introduction}

Lattice path matroids (LPM) \footnote{We use this abbreviation for lattice path matroid and lattice path matroidal depending on the context} were introduced by Bonin et.al in \cite{bonin2003lattice} and matroidal properties including the Tutte polynomial were derived for them. Subsequently, it was proven that they are positroids \cite{oh2011positroids} and also enjoy multiple connections with the positive Grassmannian. Lattice paths in themselves are ubiquitous in various topics within mathematics for example in combinatorics, representation theory, etc. In our work we see this feature helping us connect our study to various topics of not only mathematics but also to a recently defined concept in physics, the \emph{amplituhedron} \cite{arkani2014amplituhedron}, which is a geometric object encoding information concerning the scattering amplitudes of particles. 

We begin with the introduction of \emph{lattice path matroidal subdivisions}, which are matroidal subdivisions with each maximal cell corresponding to a lattice path matroid polytope. The idea for this class of subdivisions comes from the lattice path matroid polytope decompositions \cite{chatelain2011matroid}, which is a subclass of matroid base polytope decompositions, studied in detail in \cite{chatelain2011matroid,chatelain2014matroid}. Lattice path matroidal decompositions enjoy a unique property; they are obtained in an iterative way via simple decompositions into two LPMs, termed as a \emph{hyperplane split}. We harness this property to relate them to the well-known class of \emph{split subdivisions}. This relation eventually helps us in proving one of our first results.

\begin{theorem}
Any $\LPM$ subdivision of a lattice path matroid polytope $\P_{\M[P,Q]}$ is regular.
\end{theorem}

Not only we are able to establish regularity for LPM subdivisions but we also show that they are obtained as common refinements of split subdivisions, which allows much more structure to these subdivisions. We introduce the notion of LPMfan
as the polyhedral fan that corresponds to LPM subdivisions. We discuss the relation of LPMfan to various well-known polyhedral fan structures which correspond to regular matroidal subdivisions, namely \emph{tropical Grassmannian} and \emph{Dressian}. Since LPM are positroids as well, this discussion can also be connected to the positive part of the tropical Grassmannian and Dressian. We furnish computational examples for both LPM subdivisions and LPMfans for the underlying hypersimplex $\Delta(k,n)$ which is an LPM polytope, where $k=3,4$ and $n=6,8$ respectively.

Postnikov \cite{https://doi.org/10.48550/arxiv.math/0609764, postnikov2018positive} led the study on the stratification of the positive Grassmannian into cells that enjoy equivalences with various combinatorial objects, like \emph{decorated permuations}, \emph{reduced plabic graphs}, etc. We also put our results into perspective by discussing how our LPM subdivisions correspond to these combinatorial objects. This also helps us in bringing the connections to the geometric object \emph{amplituhedron}, introduced first by Arkani et. al \cite{arkani2014amplituhedron} to study problems concerning scattering amplitudes in high energy physics. We point the reader to \cite{arkani2016grassmannian} for exploring the connections between scattering amplitudes in physics and the geometry of the Grassmannian in full detail. Our discussion mostly revolves around the connections between positive Grassmannian, positive tropical Grassmannian and the amplituhedron.

Firstly, for the $m=2$ amplituhedron, we provide a purely matroidal treatment to the definition of BCFW\footnote{the abbreviation is after the names of Physicists Britto, Cachazo, Feng, and Witten} style recurrence relations for positroid dissections of the hypersimplex in the form of Theorem \ref{thm:BCFW_matroidal}. These positroidal dissections were introduced in \cite{m=2amplut} and it is shown in \cite{m=2amplut} that via \emph{T-duality} they are also related to certain dissections of the $m=2$ amplituhedron $\mathcal{A}_{n,k,2}$. Secondly, for the $m=4$ amplituhedron, in \cite{karp2020decompositions} it is shown that BCFW cells of the amplituhedron correspond to a \emph{noncrossing} lattice paths of a certain lattice rectangle. Additionally, a recent work \cite{even2021amplituhedron} shows that BCFW cells provide a triangulation of the amplituhedron $\mathcal{A}_{n,k,4}$. In light of these results, we prove the following result which is the first result highlighting the relation between the BCFW triangulation of $\mathcal{A}_{n,k,4}$ and positroidal dissection of a certain hypersimplex.

\begin{theorem}
Each triangulation of the amplituhedron $\mathcal{A}_{n,k,4}$ into $(k, n)$-BCFW cells provides a positroid dissection $\{\Gamma_{i}\}$ of the hypersimplex $\Delta(k,n-4)$, where each BCFW cell corresponds to a lattice path matroid polytope $\Gamma_{i}$.     
\end{theorem}

Lastly, \cite{arkani2021positive} discusses the relation between positroidal cells of the positive Grassmannian and the positive configuration space, via the Chow quotient of the Grassmanian. We also encounter a special class of LPM's throughout our study, namely \emph{snakes}, which are \emph{minimal}, and we use this property, to provide examples of \emph{clusters} for them, which implies intricate connections between LPM's and the underlying cluster algebra, which we wish to explore further in subsequent work. This minimality of snakes also helps us answer partially  a Question asked in \cite{olarte2019local}. We would like to make a special mention of the various salient features which we encounter for \emph{snakes} and would like to state them as follows,

\emph{Snakes are lattice path matroids, positroids, minimal, binary, indecomposable, series-parallel, graphical \cite{knauer2018tutte}, order, alcoved \footnote{We do acknowledge that order and alcoved are properties satisfied by matroid polytopes of snakes.} \cite{benedetti2023lattice}
}

In Section \ref{sec:Sec2} we introduce all basic definitions which we will use in further discussions. Section \ref{sec:Sec3} introduces the notion of LPM subdivisions and Theorem \ref{thm:LPM_is_regular} is proven here. Section \ref{sec:Sec4} describes the relation between the positive tropical Grassmannian and LPM subdivisions. Section \ref{sec:comp} collects all our computational examples, which are mostly LPM subdivisions and LPMfan for LPM polytopes $\Delta(3,6)$ and $\Delta(4,8)$. Section \ref{sec:Sec6} introduces the notion of amplituhedron and relates in detail the findings pertaining to LPM's. Finally, we discuss probable future problems and open questions in Section \ref{sec:Sec7}.

\section{Preliminaries}\label{sec:Sec2}

We would like to guide the readers unfamiliar with the concepts in this section to \cite{https://doi.org/10.48550/arxiv.math/0609764, postnikov2018positive} and \cite{maclagan2021introduction} for further details. A \emph{matroid} of rank $k$ on the set $[n] := \{1,2, \ldots, n\}$ is a nonempty collection $\M \subseteq \binom{[n]}{k}$ of $k$-element subsets of $[ n ]$, called \emph{bases} of $\M$, that satisﬁes the exchange axiom: 
For any $I , J \in \M$ and $i \in I$, there exists $j \in J$ such that $I \setminus \{ i \} \cup \{ j \} \in \M$.

A matroid is called \emph{realizable} if it can be represented by elements of a matrix over some field $\mathbb{K}$. A \emph{positroid} of rank $k$ is a matroid that can be represented by a $k \times n$-matrix with non-negative maximal
minors.

The \emph{Grassmannian} $\Gr(k,n)$ is the parameterization of the family of all $k$-dimensional subspaces of $n$-dimensional vector space in $\mathbb{K}^{n}$. It also possesses a smooth projective variety structure, corresponding to the vanishing set of the \emph{Pl\"ucker ideal} $\mathcal{I}_{k,n}$. 

An element in the Grassmannian $\Gr(k,n)$ can be understood as a collection of $n$ vectors $v_{1}, \hdots, v_{n} \in \mathbb{K}^{k}$ spanning the space $\mathbb{K}^{k}$ modulo the simultaneous action of $\GL(k,n)$ on the vectors, where the vectors $v_{i}$ are
the columns of a $k \times n$-matrix $A$. Then an element
$V \in \Gr(k,n)$ represented by $A$ gives the matroid $\M_{V}$ whose bases are the $k$-subsets $I \subset [n]$ such that $\det_{I}(A) \neq 0$. Here, $\det_{I}(A)$ denotes the determinant of $A_{I}$, the $k \times k$ submatrix of $A$ with the column set $I$.

An element $V \in \Gr(k,n)$ is termed as \emph{totally non-negative} if $\det_{I}(V) \geq 0$, for all $I \in \binom{[n]}{k}$. The set of all totally non-negative $V \in \Gr(k,n)$ is the \emph{totally non-negative Grassmannian} $\Gr^{\geq 0}(k,n)$; abusing notation, we refer to $\Gr^{\geq 0}(k,n)$ as the \emph{positive
Grassmannian} \cite{m=2amplut}.

Tropical geometry is the study of polynomials over the tropical semiring $\mathbb{T} = \{ \mathbb{R} \cup \infty, \text{max}, + \}$. Given $e = (e_{1} , \hdots , e_{N} ) \in \mathbb{Z}^{N}_{\geq 0}$ , we let $x^{e}$ denote $x_{1}^{e_{1}} \hdots x_{N}^{e_{N}}$. For a polynomial $f = \sum_{e \in E} a_{e}x^{e}$, we firstly associate a corresponding tropical polynomial $f$ where the binary operations are replaced by tropical addition and multiplication respectively, and we denote by $\text{Trop}(f)$ the \emph{tropical hypersurface} associated to $f$ which is the collection of all points where the maxima is achieved at least twice. Let $E = E^{+}  \cup  E^{-} \subseteq \mathbb{Z}^{N}_{\geq 0}$ , and let $f$ be a nonzero polynomial with real coefficients such that $f = \sum_{e \in E^{+}}a_{e}x^{e} - \sum_{e \in E^{-}}a_{e}x^{e} $, where all of the coefficients $a_{e}$ are non-negative real numbers. Then $\Trop^{+}(f)$ denotes the \emph{positive part} of $\Trop(f)$, and the set of all points $(x_{1} , \hdots , x_{N})$ such that, if we form the collection of numbers $\sum e_{i}x^{i}$ for $e$ ranging over
$E$, then the minimum of this collection is not unique and furthermore is achieved for some $e \in E^{+}$ and some $e \in E^{-}$ \cite{m=2amplut}.

The \emph{tropical Grassmannian} $\TGr(k,n)$ is the intersection of the tropical hypersurfaces $\Trop(f)$, where $f$ ranges over all elements of the \emph{Pl\"ucker ideal} $\mathcal{I}_{k,n}$ which is generated by the \emph{quadratic Pl\"ucker relations} \cite{maclagan2021introduction}. The \emph{Dressian} $\Dr(k,n)$ is the intersection of the tropical hypersurfaces $\Trop(f)$, where $f$ ranges over all three-term Pl\"ucker relations. Similarly, the \emph{positive tropical Grassmannian} $\Trop^{+}\Gr(k,n)$ is the intersection of the positive tropical hypersurfaces $\Trop^{+}(f)$, where $f$ ranges over all elements of the Pl\"ucker ideal. The \emph{positive Dressian} $\Dr^{+}(k,n)$ is the intersection of the positive tropical hypersurfaces $\Trop^{+}(f)$,
where $f$ ranges over all three-term Pl\"ucker relations. The underlying matroid for the definitions of the tropical Grassmannian and Dressian is the \emph{uniform matroid} $\U_{k,n}$. However, the notion of Dressian can be extended to arbitrary matroids with the definition of a \emph{local Dressian}. The \emph{local Dressian} $\Dr(\M)$ is defined as the tropical pre-variety  given by the set of quadrics obtained from the three-term 
Pl\"ucker relations by setting the variables $p_{B}$ to zero, where $B$ is not a basis of $\M$ \cite{olarte2019local}.

A subdivision $\Sigma$ of a polytope $P$ in $\mathbb{R}^{d}$ is said to be \emph{regular} if there exits a weight vector $w$ such that if the vertices of $P$ are lifted to heights provided by $w$ in $\mathbb{R}^{d+1}$ and subsequently the lower convex hull is projected back to $\mathbb{R}^{d}$, then the subdivision $\Sigma$ is retrieved. A tropical polynomial with Newton polytope $P$ defines a \emph{tropical hypersurface} that is dual to a regular subdivision of $P$. We point the reader to \cite[Chapter 1,3]{maclagan2021introduction}, \cite[Chapter 1]{joswig2021essentials} for further details about this duality. 

We recall details about a special class of subdivisions that appear in our work. A \emph{split} subdivision is a subdivision with exactly two maximal cells \cite{herrmannsplitting}. Two splits $S_{1}$ and $S_{2}$ are said to be \emph{compatible} if the hyperplane along the split edges do not intersect in the interior of the polytope. 

We now introduce definitions dealing with lattice path matroids. Let $E$ be a set (which is going to be the ground set of the matroid), and let $\mathcal{A}  =  (A_{j} : j \in J )$
be a set system over $E$, that is, a multiset of subsets of a ﬁnite set $S$. A \emph{transversal} of $\mathcal{A}$ is a set $\{ x_{j} : j \in J  \}$ of  $|J|$ distinct elements such that
$x_{j} \in  A_{j}$ for all $j \in J$. A \emph{partial transversal} of $\mathcal{A}$ is a transversal of a set system of the form ($A_{k} : k \in K) $ with $K$ a subset of $J$. A \emph{transversal matroid} is a matroid whose independent sets are the partial transversals of some set system $\mathcal{A} = (A_{j} : j \in J)$ and $\mathcal{A}$ is called the \emph{presentation} of the transversal matroid. We denote this matroid by $\M[\mathcal{A}]$. The bases of a transversal matroid are the maximal partial transversals of $\mathcal{A}$ \cite{bonin2003lattice}.

We now recall the definition of a lattice path matroid as a certain kind of transversal matroid \cite[Definition 3.1]{bonin2003lattice}. Consider an $r\times (n-r)$ rectangular lattice grid $\U_{r,n}$. This consists of all the lattice points $\bigg\{ (a, b) : 0 \leq a \leq n-r~,~ 0 \leq b \leq r\bigg\}$, and all the edges between neighboring lattice points. This can also be thought as a Young diagram \cite{fulton1997young} consisting of $r\cdot (n-r)$ unit squares of the partition $\lambda = (\underbrace{n-r, n-r, \cdots, n-r}_{r})$. An NE-\emph{path} over $\U_{r, n}$ is a path from the point $(0,0)$ to the point $(n-r, r)$ each of whose step is either a step in $(1,0)$ direction (i.e. an $E$-step) or a step in $(0,1)$ direction (i.e. an $N$-step). Note that for each edge in $\U_{r,n}$, its position in any NE-path is the same. Hence we can denote it by this position. Using this observation, we can denote each NE-path by the sequence of its north steps.

\begin{definition}\label{def:LPM}
Let $P$ and $Q$ be two NE-paths on $\U_{r, n}$ denoted by
\begin{align*}
    P = p_{1}p_2\ldots p_r \\
    Q = q_1 q_2 \ldots q_r
\end{align*}
then the set of all NE-paths between $P$ and $Q$ forms a matroid. That is,
\begin{align}
    \M[P,Q] &= \bigg\{\{i_1, i_2, \ldots, i_r\}: p_j \leq i_j \leq q_j~\text{for}~j=1,\ldots, r \bigg\}
\end{align}
Sometimes, we denote the matroid $\M[P,Q]$ by just $\M[J]$ where $J$ is the skew Young diagram bounded by $P$ and $Q$.
\end{definition}


An example of a lattice path matroid is depicted in Figure \ref{fig:LPM_example}, where the edges in the North direction are marked with their respective indices.

\begin{figure}[H]
    \centering
    \includegraphics[scale=0.5]{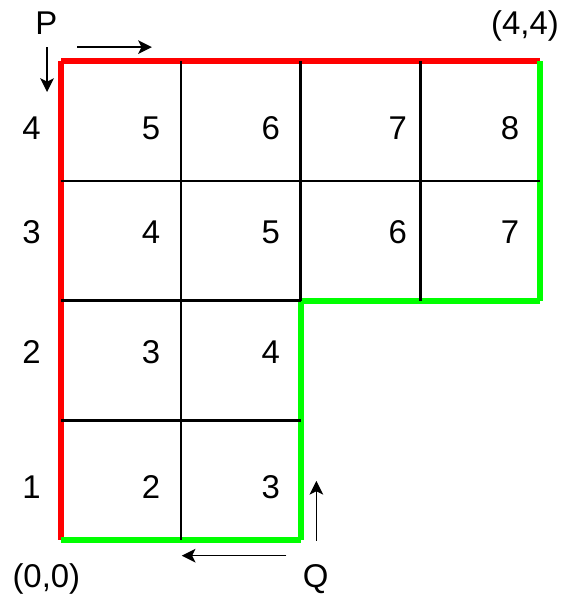}
    \caption{The lattice path in red depicts the path $P$ and the lattice path in green depicts the path $Q$, for the lattice path matroid $\M[P,Q]$. }
    \label{fig:LPM_example}
\end{figure}

\section{LPM Subdivisions}\label{sec:Sec3}

We use the following definition from \cite{knauer2018tutte},

\begin{definition}
We call a lattice path matroid $\M[P,Q]$ a \emph{snake} if it has at least two elements, it is connected and the strip contained between the paths $P$ and $Q$ does not contain any interior lattice point.   
\end{definition}

Snakes are also referred to as \emph{border strip} matroids. Snakes have the minimal number of bases a rank $r$ connected matroid over $n$ elements has. That is why they are also called \emph{minimal} matroids \cite{ferroni2022ehrhart}. In contrast to this, uniform matroids are \emph{maximal} with respect to this property.

We introduce a new class of subdivision as follows:
\begin{definition}
Let $\M[P,Q]$ be a lattice path matroid and $\P_{\M[P,Q]}$ be its matroid polytope.
A subdivision $\Sigma$ of  $\P_{\M[P,Q]}$ is called a \emph{lattice path matroidal} (LPM) subdivision if all maximal cells of $\Sigma$ are lattice path matroid polytopes.
\end{definition}

In \cite{chatelain2011matroid, ferroni2022valuative}, matroid base polytope decompositions are studied in detail and it is shown that for a lattice path matroid $\M[P,Q]$ which is not a \emph{snake}, its matroid polytope $\P_{\M[P,Q]}$ admits a decomposition into lattice path matroids polytopes s.t 
\begin{align}
\P_{\M[P,Q]} = \bigcup_{i=1}^{t} \P_{\M[P_i,Q_i]} 
\end{align}

where each $\P_{\M[P_{i},Q_{i}]}$ is also a lattice path matroid base polytope for some lattice path matroid $M[P_{i},Q_{i}]$, and for each $1 \leq i \neq j \leq t$,
the intersection $\P_{\M[P_{i},Q_{i}]} \cap \P_{\M[P_{j},Q_{j}]}$ is a face of both $\P_{\M[P_{i},Q_{i}]}$ and $\P_{\M[P_{j},Q_{j}]}$. A \emph{hyperplane LPM split} decomposition is a decomposition in exactly two lattice path matroid polytopes, i.e., $t=2$ and as a consequence of \cite[Corollary 1]{chatelain2011matroid} we also know that these two LPM polytopes are full-dimensional.

We feel that it is a good time to recall the notion of a polytopal subdivision \cite[Section 1.2]{joswig2021essentials}, 

\begin{definition}
For a polytope  $P \in \mathbb{R}^{d}$, a \emph{(polyhedral) subdivision} $\Sigma$ is a polytopal complex whose vertices are the vertices of $P$, and that covers $P$. $\Sigma$ can be understood as a collection of faces $F$, such that for any two faces $F_{i}$ and $F_{j}$, $F_{i} \cap F_{j} \in \Sigma$.      
\end{definition}

It is pretty obvious from the definition above that the notions of LPM decompositions and LPM subdivisions coincide and we state this in the form of Corollary \ref{cor:subdiv_eq_decomp} 

\begin{corollary}\label{cor:subdiv_eq_decomp}
Let $\Sigma' = (\P_{\M[P_{1},Q_{1}])}, \hdots \P_{\M[P_t,Q_t]})$ be a decomposition of $\P_{\M[P,Q]}$ into lattice path matroid polytopes. Then $\Sigma'$ coincides with the subdivision $\Sigma$ of $\P_{\M[P,Q]}$, where each maximal cell is $C_{i} =  \P_{\M[P_{i},Q_{i}]}$.   
\end{corollary}

The LPM subdivision corresponding to a hyperplane LPM split decomposition is called a split subdivision. The subsequent subdivision is obtained iteratively via split subdivisions which correspond to hyperplane LPM split decomposition. We take this opportunity to specify our terminology so as to minimize any confusion in the text; split with the prefix 'hyperplane' would always refer to the LPM subdivision of a lattice path matroid into two LPM, whereas split with the suffix 'hyperplane' would refer to the hyperplane defining a split subdivision.    

From now on our discussion would mostly focus on the LPM subdivisions, however, because of the equivalence in Corollary \ref{cor:subdiv_eq_decomp} most of our results also extend to LPM decompositions, unless otherwise stated. 

\begin{remark}  
The property of being obtained via iterative hyperplane LPM split decompositions  is unique to the LPM decompositions described in \cite{chatelain2011matroid, bidkhori2012lattice} and is different in this aspect from the concept of matroid decompositions defined in \cite{billera2009quasisymmetric} which define a new quasisymmetric invariant for matroids which acts as a valuation on decompositions of matroid polytopes. Although, Kapranov \cite{kapranov1992chow} showed that for rank 2 matroids such matroid decompositions can be obtained via hyperplane split decompositions.
\end{remark}

We recall first this technical result regarding split subdivisions,

\begin{lemma}[Lemma 3.5 \cite{herrmannsplitting}]\label{lem:split_regular}
Split subdivisions are regular.    
\end{lemma}

\begin{proof}
Let $S$ be a split subdivision of a polytope $P$. We provide a canonical weight vector for this subdivision in the following way. Let $a$ be the normal vector to the split hyperplane $H_{S}$. We define the weight vector for $S$ as $w_{S}: \text{Vert}(P) \rightarrow \mathbb{R}$ such that 
\begin{align*}
    w_{S}(v) = \begin{cases}
    |av| \quad \text{if} \> v \in S_{+} \\ 
    0 \quad \quad \text{if} \> v \in S_{-}
    \end{cases}
\end{align*}

It is clear that this weight function is well-defined and induces the split subdivision $S$.
\end{proof}

We now state a technical result concerning split LPM subdivisions. We call an LPM polytope $\P_{\M[J]} \subseteq \P_{\M[P,Q]}$ a \emph{truncated} LPM polytope if $\P_{\M[J]}  = \P_{\M[P,Q]} \setminus (\P_{\M[P,Q]} \cap H_{-})$, where $H_{-}$ is a halfspace defined by the split hyperplane $H$ of a split subdivision (cf. Figure \ref{fig:trunc}).

\begin{lemma}\label{lem:split_ext}
A split subdivision of a truncated LPM polytope $\P_{\M[J]}$ into two LPM can be extended to a split subdivision of the LPM polytope $\P_{\M[P,Q]}$ into two LPM.      
\end{lemma}

\begin{proof}
 We consider  a split $S$ of the LPM polytope $\P_{\M[P,Q]}$. By Lemma \ref{lem:split_regular}  we know there exists a weight vector $w_{S}$ of the form 
\begin{align*}
    w_{S}(v) &= 
    \begin{cases}
        |av| \quad \text{if} \> v \in S_{+} \\ 
        0 \quad \quad \text{if} \> v \in S_{-}
    \end{cases}
\end{align*}
where $a$ is the normal vector to the split hyperplane $H_{S}$.
Similarly, let us consider a split $S'$ of the truncated LPM polytope 
$\P_{\M[J]}$. Again by Lemma \ref{lem:split_regular} we know that restricted to $\P_{\M[J]}$ there exists a weight vector $w_{S'}$ of the form 
\begin{align*}
    w_{S'}(v) &=
    \begin{cases}
    |bv| \quad \text{if} \> v \in S'_{+} \\ 
    0 \quad \quad \text{if} \> v \in S'_{-}
    \end{cases}
\end{align*}
where $b$ is the normal vector the split hyperplane $H_{S'}$ and we choose $S'_{-}$ such that $S_{-} \subseteq S'_{-}$. Now we notice that there exists an extension of the weight vector $w_{S'}$ to $w'_{S'}$ which is defined as follows 
\begin{align*}
    w'_{S'}(v) &=
   \begin{cases}
  w_{S'}(v) \quad &\text{if} \> v \in \P_{\M[J]} \\ 
  0 \quad \quad &\text{if} \> v \in \P_{\M[P,Q]} \cap S_{-}
  \end{cases}
\end{align*}
\end{proof}

\begin{lemma}\label{lem:split_compatible}
For an LPM polytope $\P_{\M[P,Q]}$, the split subdivisions induced from a hyperplane split decomposition are compatible.    
\end{lemma}

\begin{proof}
We proceed by proving the claim for two arbitrarily chosen split subdivisions. Let $S_{1}$ and $S_{2}$ be two split subdivisions of $\P_{\M[P,Q]}$. Since split LPM subdivisions are defined in an iterative manner, therefore without loss of generality we assume that $S_{2}$ restricted to the truncated LPM polytope $\P_{\M[J]}  = \P_{\M[P,Q]} \setminus (\P_{\M[P,Q]} \cap S_{1_{-}})$ defines a split subdivision for $\P_{\M[J]}$. But this implies that the split hyperplane $H_{S_{2}}\in \P_{\M[J]}$. Therefore, the split hyperplane $H_{S_{1}}$ and $H_{S_{2}}$ cannot meet in the interior of $\P_{\M[P,Q]}$. Hence, the splits $S_{1}$ and $S_{2}$ are compatible.   
\end{proof}

\begin{remark}
As for the case of the hypersimplex $\Delta(k,n)$ which is also a LPM polytope, we already know that any two splits are always compatible \cite[Corollary 5.6]{herrmannsplitting}.    
\end{remark}

\begin{remark}
The compatibility of splits which provide the iterative description of LPM subdivisions also shows that LPM are \emph{split matroids}, introduced by Joswig and Schroeter in \cite{joswig2017matroids}.
\end{remark}

\begin{figure}[H]
    \centering
    \includegraphics[scale=0.33]{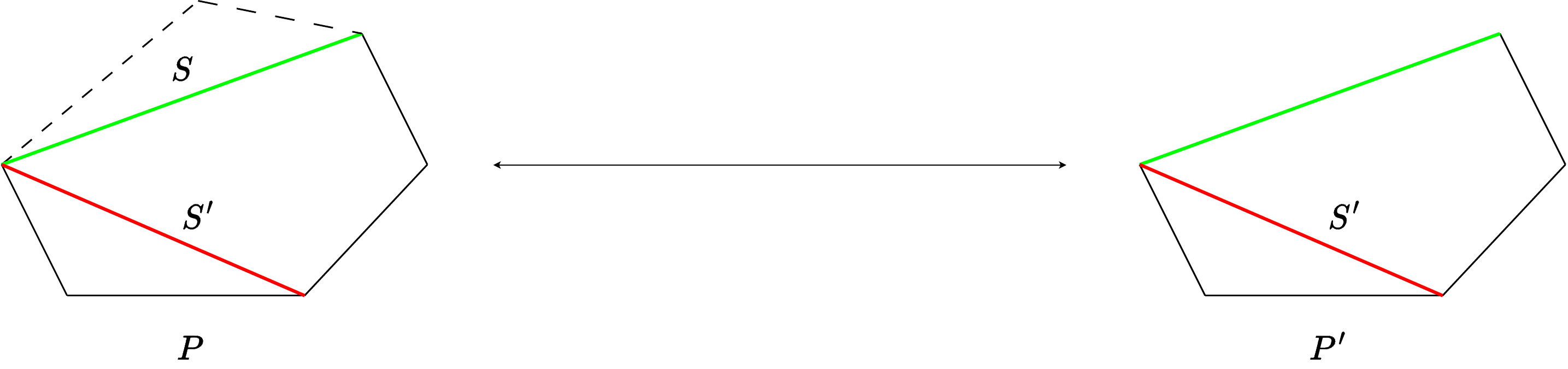}
    \caption{A pictorial description of a truncated polytope $P'$ from $P$ with respect to a split $S$, where $S'$ is a split of the truncated polytope $P'$ which can be extended to the polytope $P$}
    \label{fig:trunc}
\end{figure}

\begin{theorem}\label{thm:LPM_is_regular}
Any LPM subdivision $\Sigma$ of a lattice path matroid polytope $\P_{\M[P,Q]}$ is regular.
\end{theorem}

\begin{proof}

Let $\sigma$ be the LPM decomposition corresponding to $\Sigma$. We know that $\sigma$ can be obtained via iterative hyperplane LPM split decompositions. These hyperplane LPM split decompositions correspond to split subdivisions. Let $\{S_{1}, S_{2}, \hdots S_{n}\}$ be the sequence of split subdivisions which correspond to $\Sigma$. We note that $\{S_{2},\hdots, S_{n}\}$ are splits for the corresponding truncated LPM polytope $\P_{\M[J]}$. By Lemma \ref{lem:split_ext} we know that the splits $\{ S_{2}, S_{3}, \hdots S_{n} \}$ can be extended to split subdivisions for $\P_{\M[P,Q]}$ and let $\{S'_{2}, S'_{3}, \hdots S'_{n}\}$ be the corresponding split subdivisions on $\P_{\M[P,Q]}$ for $\Sigma$. We see that $\Sigma$ is the common refinement of the splits  $\{S_{1}, S'_{2}, \hdots S'_{n}\}$ and as we know from Lemma \ref{lem:split_compatible} that these splits are compatible, therefore this common refinement is well defined. We now invoke the Split Decomposition Theorem \cite[Theorem 3.10]{herrmannsplitting} to conclude that there exists a canonical weight vector 

\[ w = \sum_{S'} \alpha^{w}_{w_{S'}} w_{S'} \]

which induces $\Sigma$, where the sum runs over all splits and $\alpha^{w}_{w_{S}}$ represents the \emph{coherency index} \cite{herrmannsplitting}. Hence, $\Sigma$ is a regular subdivision.
\end{proof}

\begin{example}
For the hypersimplex $\Delta(3,6)$ we describe a LPM subdivision $\Sigma^{\LPM}$ in Section \ref{sec:comp}, illustrated in Figure \ref{fig:LPM_(3,6)}, where we see that in the corresponding LPM polytope decomposition $(M_{1} , \hdots , M_{6})$ shown in Figure \ref{fig:LPM_(3,6)_decomp}, is obtained as common refinements of four splits namely $S_{1}, S_{2}, S_{3}$ and $S_{4}$. The weight which induces  
$\Sigma^{\LPM}$ is 

\[ w_{\Sigma^{\LPM}} = \{0,0,0,0,0,0,0,1,1,2,0,0,0,1,1,2,2,2,3,5\} \]

and 

\[ w_{S_{1}} = \{0,0,0,0,0,0,0,1,1,1,0,0,0,1,1,1,1,1,1,2\} \]
\[ w_{S_{2}} = \{0,0,0,0,0,0,0,0,0,1,0,0,0,0,0,1,0,0,1,1\} \]
\[ w_{S_{3}} =  \{0,0,0,0,0,0,0,0,0,0,0,0,0,0,0,0,1,1,1,1\} \]
\[ w_{S_{4}} = \{0,0,0,0,0,0,0,0,0,0,0,0,0,0,0,0,0,0,0,1\}  \]

are the weights which induce the splits $S_{1}, S_{2}, S_{3}$ and $S_{4}$. With this, we see an example of the result described in Theorem \ref{thm:LPM_is_regular}, with the split decomposition in the following form, 

\[ w_{\Sigma^{\LPM}} = w_{S_{1}} + w_{S_{2}} + w_{S_{3}} + w_{S_{4}} \]
\end{example}

\begin{corollary}\label{cor:LPM_in_Dressian}
Let $w_{\Sigma}$ be a weight vector for an LPM subdivision $\Sigma$ of a lattice path matroid polytope $\P_{\M[P,Q]}$. Then $w_{\Sigma} \in \Dr(\M[P,Q])$.
\end{corollary}

\begin{proof}
Since $w_{\Sigma}$ induces a regular and matroidal subdivsion $\Sigma$, therefore by \cite[Corollary 4.1]{olarte2019local} $w_{\Sigma}$ lies in the Dressian $\Dr(\M[P,Q])$.
\end{proof}

We know that the Dressian is endowed with two polyhedral fan structures; one coming from the tropical prevariety definition with points satisfying Pl\"ucker relations and termed as the \emph{Pl\"ucker fan structure} \cite{olarte2019local} on the Dressian. The other structure termed as the \emph{secondary fan structure} \cite{olarte2019local} which comes by virtue of being a subfan of the secondary fan. Moreover, we know that these two fan structures coincide \cite[Theorem 4.1]{olarte2019local}. We now have the required setup to describe a new polyhedral fan structure for LPM subdivisions. We begin this exploration with the following definition,

\begin{definition}
Let $\M[P,Q]$ be a lattice path matroid. We define the $\text{LPMfan}(\M[P,Q])$ to be the polyhedral fan which is the collection of all weight vectors $w$ such that, $w$ is a weight vector for an LPM subdivision of $\M[P,Q]$. Two weight vector $w_{1}$ and $w_{2}$ lie in the same cone $C$ if the LPM subdivisions $\Sigma_{1}$ and $\Sigma_{2}$ are same.     
\end{definition}

Clearly, 

\begin{equation}\label{eq:LPM_cont_1}
   \text{LPMfan}(\M[P,Q]) \subseteq \text{Dr}(\M[P,Q]) \subseteq \text{Secfan}(P_{\M[P,Q]})   
\end{equation}

where all inclusions represent inclusions as subfan. Additionally, from the definition of LPM subdivisions, given that they are obtained via refinement of split subdivisions, this makes the LPMfan sit as a subfan inside the \emph{split complex}  $\text{Split}(P_{\M[P,Q]})$, which is an abstract simplicial complex defined on the set of compatible splits of $\P_{\M[P,Q]}$ \cite{herrmannsplitting}.  
Hence, we get this refined containment relation of subfans,

\begin{equation}\label{eq:LPM_cont_2}
   \text{LPMfan}(\M[P,Q])) \subseteq \text{Split}(P_{\M[P,Q]}) \subseteq \text{Dr}(\M[P,Q]) \subseteq \text{Secfan}(P_{\M[P,Q]})   
\end{equation}

An important observation is that the hypersimplex $\Delta(k,n)$ is a lattice path matroid polytope and hence all our results for LPM polytopes follow in this case, 

\[ \text{LPMfan}(k,n) \subseteq \text{Split}(\Delta(k,n)) \subseteq \Dr(k,n) \subseteq \text{Secfan}(\Delta(k,n)) \]

An important avenue of research has been to understand the structure of the Dressian $\Dr(k,n)$, particularly for certain low values of $k$ and $n$, namely $(3,6), (3,7),$ \cite{herrmann2009draw} and $(3,8)$ \cite{herrmann2014dressians},  etc. We describe LPMfans for certain values of $k,n$ and discuss the calculations in Section \ref{sec:comp}. 

\section{Positive Tropical Grassmannian and LPM subdivisions}\label{sec:Sec4}
 
In this section, our aim is to highlight the consequences of the fact that LPMs are positroids, and towards the end we also are able to provide an answer to a question asked concerning finest matroidal subdivisions of the hypersimplex in \cite{olarte2019local}. Since it is a major theme for this section we recall the result from \cite{oh2011positroids} which shows us that lattice path matroids are positroids, upon which we build further in this section.

\begin{theorem}[Lemma 23 \cite{oh2011positroids}]
 A lattice path matroid is a positroid.   
\end{theorem}

\begin{proof}
Let $\M[P,Q])$ be a LPM. For the result to be true, it is sufficient to construct a $k \times n$ matrix $A$ such that
\begin{align*}
    \det(A_I) &=
    \begin{cases}
    0 \quad \quad \quad \quad &I \in \binom{[n]}{k}  \setminus {\M[P,Q]} \\ 
    \alpha  \quad   &I \in \M[P,Q]
    \end{cases}
\end{align*}
where $\alpha > 0$. Such a matrix can be constructed as follows. Let $A = (a_{i,j})^{k,n}_{i,j=1,1}$ be the $k \times n$ \emph{Vandermonde} matrix. Set $a_{i,j} = 0 \>\> \forall \>\> j \in [P_{i},Q_{i}]$, where $P_{i}$ and $Q_{i}$ represent the $i^{th}$ north step in the lattice paths $P$ and $Q$, respectively. So $A$ has the following form,
\begin{align}
    a_{i,j} &= 
  \begin{cases}
  x_{i}^{j-1} \quad \quad  \text{if} \> P_{i} \leq j \leq  Q_{i} \\ 
  0  \quad \quad \quad  \> \text{otherwise}
  \end{cases}
\end{align}
Assign values to variables $x_{1}, \hdots , x_{k}$ such that $x_{1} > 1$ and $x_{i+1} = x_{i}^{k^{2}} \forall i \in [k-1]$. We denote the submatrix $A_{[1, \hdots , i][c_{1}, \hdots, c_{i}]}$ as a submatrix of $A$ which has rows indexed from 1 to $i$ and columns indexed from $c_{1}$ to $c_{i}$. We have $\det(A_I) > 0$ if and only if $A_{[1, \hdots, k]I}$ has nonzero diagonal entries, which happens if and only if $I \in \M[P,Q])$.
\end{proof}

Lusztig \cite{Lusztig1994} and Postnikov\cite{https://doi.org/10.48550/arxiv.math/0609764,postnikov2018positive} introduced the notion of positivity for Grassmannians. This notion extends naturally to the tropical Grassmannian and Dressian \cite[Section 4.3]{maclagan2021introduction}. In \cite{speyer2021positive} and independently in \cite{arkani2021positive}, the authors prove the following equality between  $\Trop^{+} Gr(k,n)$ and $Dr^{+}(k,n)$.

\begin{theorem}[Theorem 3.9 \cite{speyer2021positive}]
 The positive tropical Grassmannian  $\Trop^{+}\Gr(k,n)$ equals the positive Dressian $\Dr^{+}(k,n)$. 
\end{theorem}


A generalization of this theorem to the case of positive local Dressian with respect to a positroid $\M$ is provided in \cite{arkani2021positive}. An important parameterization of points residing in the positive Dressian is explained in this result,

\begin{theorem}[Theorem 4.3 \cite{speyer2021positive}]
 Let $\Sigma$ be a regular subdivision of $\Delta(k,n)$ induced by a weight vector $w_{\Sigma}$. Then the following are equivalent:
 \begin{enumerate}
     \item[1.] $w$ is a positive tropical Pl\"ucker vector.
    \item[2.] Every face of $\Sigma$ is a positroid.
 \end{enumerate}
\end{theorem}

The generalization of this to local positive Dressian is provided again in \cite[Proposition 8.3]{arkani2021positive}. With this parameterization, we conclude that a point inducing a LPM subdivision resides in the positive Dressian.

\begin{lemma}\label{lem:lpmsubposdress}
Let $\Sigma$ be an LPM subdivision of $\P_{\M[P,Q]}$ and let $w_{\Sigma}$ be the weight vector for $\Sigma$. Then $w \in \Dr^{+}(\M[P,Q]) =  \Trop^{+}\Gr(\M[P,Q])$. 
\end{lemma}

\begin{proof}
We know that a point $w$ lies in the positive Dressian if all the  maximal cells of the subdivision induced by this point as a weight vector on $\P_{\M[P,Q]}$ are the matroid polytopes of a positroid, i.e., $w$ induces a \emph{positroidal} subdivision \cite[Proposition 8.3]{arkani2021positive}. We know that LPM are positroids, hence it also induces a positroidal subdivision, and therefore $w \in \Dr^{+}(\M[P,Q]) =  \Trop^{+} \Gr(\M[P,Q])$.  
\end{proof}

Another important result proven in \cite{speyer2021positive} is about the classification of the finest positroidal subdivision of the hypersimplex $\Delta(k,n)$.

\begin{theorem}\label{thm:posit_clasif}
Let $\Sigma$ be a regular positroidal subdivision of $\Delta(k,n)$. Then the following are equivalent:
\begin{enumerate}
    \item[1.] $\Sigma$ is a finest subdivision.
    \item[2.] Every facet of $\Sigma$ is the matroid polytope of a  series-parallel matroid.
    \item[3.] Every octahedron in $\Sigma$ is subdivided.
\end{enumerate}
\end{theorem}

Along with the classification, \cite{speyer2021positive} also provides the exact number of maximal cells in a finest positroidal subdivision of $\Delta(k,n)$,

\begin{corollary}
Every finest positroidal subdivision of $\Delta(k,n)$ has exactly $\binom{n-2}{k-1}$ facets.
\end{corollary}

We also recall the following classification of connected positroids which are series-parallel,

\begin{lemma}\label{cor:series_paraller_minor}
A connected positroid is series-parallel if and only if it has no
uniform matroid $\U_{2,4}$ as a minor.   
\end{lemma}

In light of these results, we provide results about positroidal subdivisions of $\Delta(k,n)$ obtained from LPM.  We begin with our first technical result concerning snakes,

\begin{lemma}\label{lem:Snakes_are_series-parallel}
Snakes are series-parallel matroid.    
\end{lemma}

\begin{proof}
We acknowledge that the uniform matroid $\U_{2,4}$ is also an LPM as shown in Figure \ref{fig:LPM_(2,4)}.

\begin{figure}[H]
    \centering
    \includegraphics[scale=0.7]{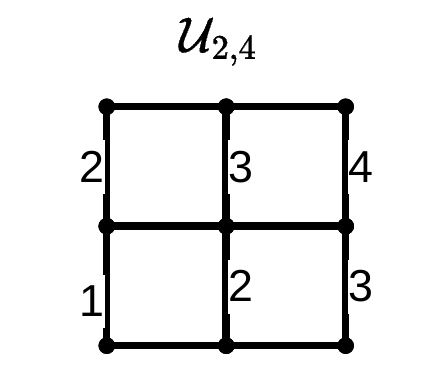}
    \caption{Uniform matroid $\U_{2,4}$ as a lattice path matroid}
    \label{fig:LPM_(2,4)}
\end{figure}

Clearly, $\U_{2,4}$ has an interior lattice point and therefore cannot be a minor of a lattice path matroid which is a snake. Therefore, by Lemma \ref{cor:series_paraller_minor} snakes are series-parallel matroid.

\end{proof}

We also acknowledge that another proof of this result is present in \cite[Proposition 5.14]{ferroni2022valuative}. With Lemma \ref{lem:Snakes_are_series-parallel} and Theorem \ref{thm:posit_clasif}, we state the following result

\begin{corollary}
Let $\Sigma$ be an LPM subdivision of $\Delta(k,n)$ such that the underlying matroid of each maximal cell is a snake. Then, $\Sigma$ is a finest positroidal subdivision of $\Delta(k,n)$ and has exactly $\binom{n-2}{k-1}$ facets.    
\end{corollary}

With Theorem \ref{thm:posit_clasif} and Lemma \ref{cor:series_paraller_minor}, we also are able to provide a partial answer to Question $6.2$ posed in \cite{olarte2019local},

\begin{question}[Question 6.2 \cite{olarte2019local}]\label{ques:indecompose}
Are all cells in the finest matroid subdivision of a hypersimplex, matroid polytopes of indecomposable matroids?
\end{question}

The authors show that the answer to this question is affirmative in the case when the hypersimplex is $\Delta(2,n)$ \cite[Proposition 5.3]{olarte2019local}. However, we know of explicit counterexamples provided in \cite{brandt2021tropical} which show that there exist finest matroidal subdivisions of certain hypersimplices, whose cells do not correspond to indecomposable matroids.

We state some technical definitions before stating the partial answer. We state the following classification for \emph{binary matroids}; which are matroids representable over the field with two elements.

\begin{theorem}[Tutte \cite{tutte1958homotopy}]
A matroid is said to be binary if and only if it has no minor isomorphic to the uniform matroid $\U_{2,4}$.   
\end{theorem}

\begin{definition}[Definition 5.2 \cite{olarte2019local}]
A matroid is said to be \emph{indecomposable} if and only if its polytope does not allow a non-trivial matroid subdivision.    
\end{definition}

Therefore, we obtain Corollary \ref{cor:indecompose} as an answer to Question \ref{ques:indecompose}, when restricted to the case of positroidal subdivisions of the hypersimplex.

\begin{corollary}\label{cor:indecompose}
The cells of the finest positroidal subdivision of $\Delta(k,n)$ correspond to binary matroids. In particular, they are indecomposable.    
\end{corollary}

\begin{proof}
We know from  Theorem \ref{thm:posit_clasif} that maximal cells of the finest positroidal subdivisions of $\Delta(k,n)$ correspond to connected series-parallel positroids and by Lemma \ref{cor:series_paraller_minor} we know that they do not have $U_{2,4}$ as a minor and therefore are also binary matroids.    
\end{proof}

With Lemma \ref{lem:lpmsubposdress} it is clear that the corresponding fan structure for LPM subdivisions also resides as a subfan inside the positive Dressian

\begin{equation}\label{eq:subfan_LPM_pos__local_dres}
\text{LPMfan}(\M[P,Q])) \subseteq \text{Dr}^{+}(\M[P,Q])) =  \text{Trop}^{+}Gr(\M[P,Q]))
\end{equation}

\begin{equation}\label{eq:subfan_LPM_pos_dres}
\text{LPMfan}(\Delta(k,n)) \subseteq \text{Dr}^{+}(k,n) =  \text{Trop}^{+} Gr(k,n)
\end{equation}

Also, in \cite{arkani2021positive} a third fan structure on the positive Dressian ${\text{Dr}^{+}}(\M)$ is defined as the \emph{positive fan structure}. This fan structure is based on the underlying \emph{cluster algebra}, studied in detail in \cite{arkani2021positive}. We refer the reader to \cite{fomin2021introduction, marsh2014lecture} for basic details concerning cluster algebras. Our aim here is to highlight the third fan structure on the positive Dressian that is induced via these clusters, although they will emerge later again in our discussion concerning minimal positroids and positive configuration space in Section \ref{subsec:pos_conf_space}. We define the notion of a \emph{cluster} associated with a matroid \cite{arkani2021positive},

\begin{definition}\label{def:cluster}
A cluster $\mathcal{C}$ for a matroid $\M$ is a subset of $\M$ that indexes a seed in the cluster structure of the cluster algebra isomorphic to $\mathbb{C}[\Tilde{\pi}_{\M}]$, where $\mathbb{C}[\Tilde{\pi}_{\M}]$ is the coordinate ring associated to the positroid variety.     
\end{definition}

\begin{definition}
The \emph{positive fan structure} on ${\text{Dr}^{+}}(\M)$ is the fan whose cones are the images of the domains of linearity for a positive parameterization by a cluster $\mathcal{C}$. Two points lie in the same cone of ${\text{Dr}^{+}}(\M)$, if they determine the same common domains of linearity for all the functions $p_{J}, J \in \M$. 
\end{definition}

The authors in \cite{arkani2021positive} also prove that this new fan structure coincides with the previous two fan structures

\begin{theorem}[Theorem 10.3 \cite{arkani2021positive}]
The three fan structures on ${\Dr^{+}}(\M)$ coincide.    
\end{theorem}

With the sub-fan relation in place \ref{eq:subfan_LPM_pos__local_dres} 

\begin{corollary}
The three fan structures on $\text{LPMfan}(\P_{\M[P,Q]})$ coincide.    \end{corollary}

\begin{remark}
We also want to highlight that matroid decompositions are invariant under matroid duality, which is also reflected in our description of the LPMfan, meaning if a $k$-dimensional cone $C$ in the LPMfan$(\P_{\M[P,Q]})$ corresponds to an LPM decomposition $\M_{t}[P^{t},Q^{t}]$, then there exists a $k$-dimensional cone $C'$ such that it represents the LPM decomposition $\{\M_{t}^{*}[P^{t},Q^{t}]\}$, where $*$ represents the matroid dual. This fact can be verified in the case of $\Delta(3,6)$ from Figure \ref{fig:LPM_fan}.

\end{remark}

\section{Computations for LPM polytope \texorpdfstring{$\Delta(k,n)$} {TEXT} }\label{sec:comp}

In this section we look at some computational examples, concentrating on the case of $\Delta(k,n)$ for $k=3,4$ and $n=6,8$ respectively. We use \texttt{polymake} \cite{gawrilow2000polymake} for our computations. 

\subsection{Computations for LPM polytope $\Delta(3,6)$}

\begin{figure}
    \centering
    \includegraphics[scale=0.3]{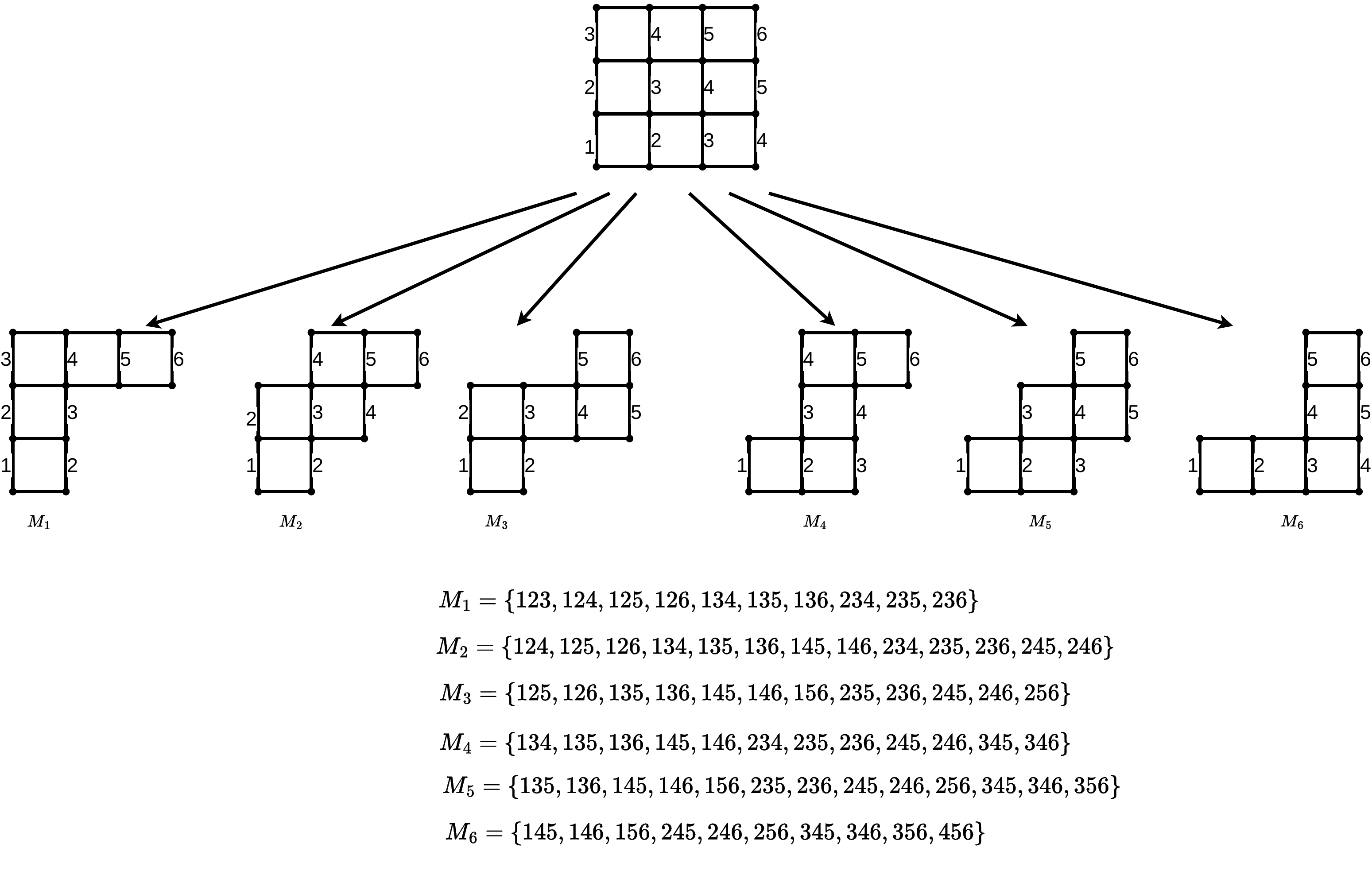}
    \caption{A LPM subdivision of $\Delta(3,6)$, which is also one of the finest  matroidal subdivisions, and hence indexes a maximal cone of $Dr(3,6)$.}
    \label{fig:LPM_(3,6)}
\end{figure}

Figure \ref{fig:LPM_(3,6)} illustrates an LPM  subdivision $\Sigma^{\LPM}$ of $\Delta(3,6)$ with the lattice path matroids corresponding to the maximal cells also shown. We also calculate the weight vector $w$ which induces this subdivision 

\[  w = \{ 0,0,0,0,0,0,0,1,1,2,0,0,0,1,1,2,2,2,3,5\}  \]

We illustrate the LPM polytope decomposition which corresponds to the subdivision in Figure \ref{fig:LPM_(3,6)} and we see the truncated LPM polytope after each iterative step of taking a hyperplane split decomposition in Figure \ref{fig:LPM_(3,6)_decomp}. 

\begin{figure}
    \centering
    \includegraphics[scale=0.1463]{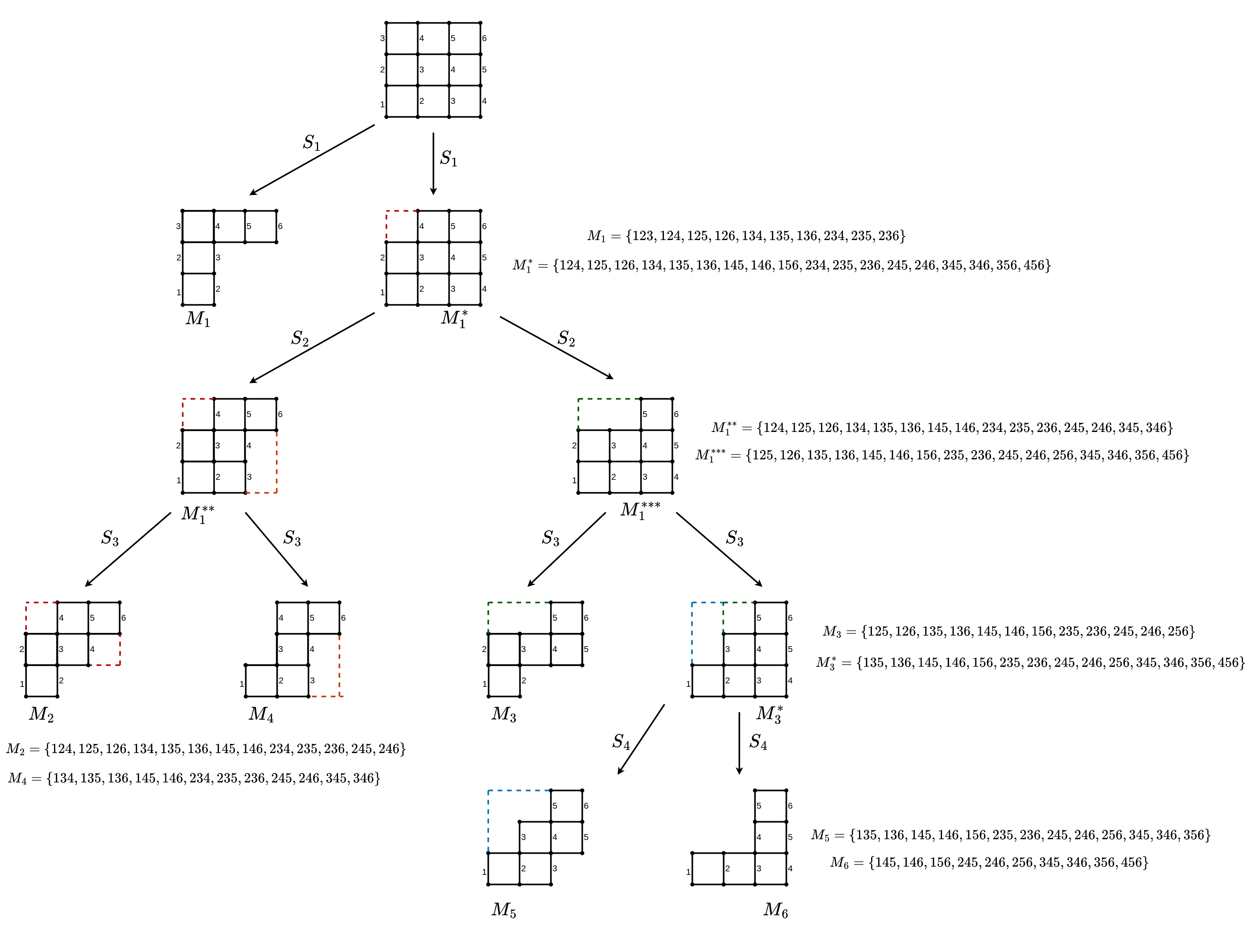}
    \caption{The LPM decomposition of $\U_{3,6}$, which corresponds to the LPM subdivision in Figure \ref{fig:LPM_(3,6)}. The dashed portions show the truncated parts of $\U_{3,6}$.}
    \label{fig:LPM_(3,6)_decomp}
\end{figure}

We also see that $\Sigma^{LPM}$ corresponds to a metric tree arrangement shown in Figure \ref{fig:LPM_(3,6)_metric_tree}.

\begin{figure}[H]
    \centering
    \includegraphics[scale=0.61]{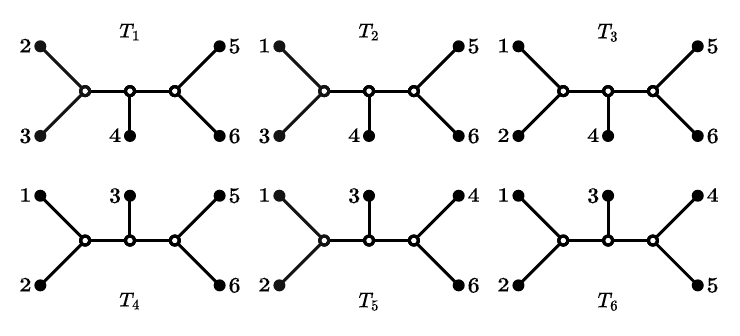}
    \caption{Metric tree arrangement corresponding to the LPM subdivision in Figure \ref{fig:LPM_(3,6)}}
    \label{fig:LPM_(3,6)_metric_tree}
\end{figure}

It is easy to see that under the permutation

\[ 1 \rightarrow 1,\quad 2 \rightarrow 5, \quad 3 \rightarrow 3, \quad 4 \rightarrow 2, \quad 5 \rightarrow 4, \quad 6 \rightarrow 6  \]

this tree arrangement permutes to the tree arrangement shown in Figure \ref{fig:Dr_(3,6)_Cone_4} which corresponds to Cone$_{4}$ \cite{tewari2022generalized} in the classification of all maximal cones of $Dr(3,6)$ \cite{herrmann2009draw}. 

\begin{figure}[H]
    \centering
    \includegraphics[scale=0.65]{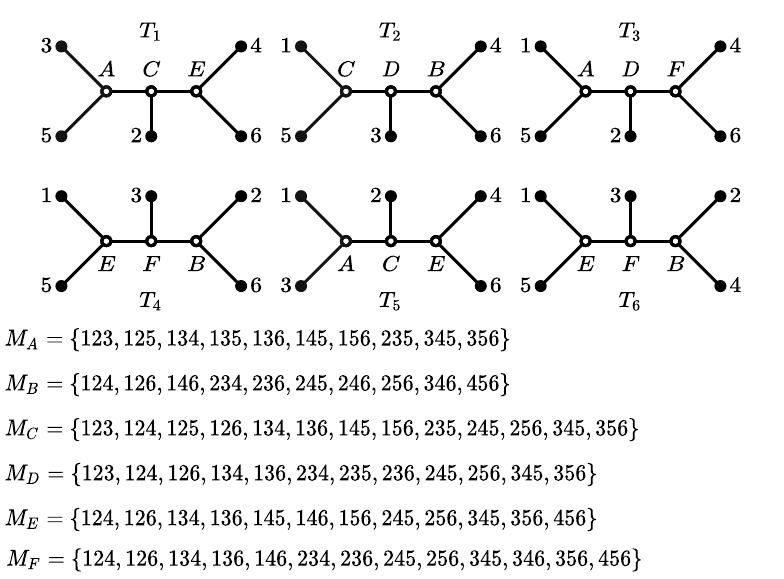}
    \caption{Metric tree arrangement and the matroidal subdivision corresponding to Cone$_{4}$ in $\Dr(3,6)$.}
    \label{fig:Dr_(3,6)_Cone_4}
\end{figure}

\subsection{Decorated permutations and reduced plabic graphs}

We now connect our computations to some other parameterization of the positive Grassmannian, namely \emph{decorated permutations} and \emph{reduced plabic graphs}, and we rely on \cite{m=2amplut} for most of our definitions in this subsection.

\begin{definition}
A \emph{decorated permutation} of $[n]$ is a bijection $\pi : [n] \rightarrow [n]$ whose fixed points are each colored either black or white. A black fixed point $i$ is denoted by $\pi(i) = \underline{i}$, and a white fixed point $i$ by $\pi(i) = \overline{i}$. An \emph{anti-excedance} of the decorated permutation $\pi$ is an element $i \in [n]$ such that either $\pi^{-1}(i) > i$ or $\pi(i) = i$. A decorated permutation on $[n]$ is of type $(k, n)$ if it has $k$ anti-excedances.    
\end{definition}

We now establish the connection between decorated permutations and positroid cells of the positive Grassmanians.

\begin{definition}
Given a $k \times n$ matrix $C = (c_{1} , \hdots , c_{n})$ written as a list of its columns,  a decorated permutation $\pi := \pi_{C}$ is associated to $C$ as follows. Set $\pi(i) := j$ to be the label of the first column $j$ such that $c_{i} \in \text{span} \> {c_{i+1} , c_{i+2} , \hdots , c_{j}}$. If $c_{i}$ is the all-zero vector, it is called a \emph{loop} and if $c_{i}$ is not in the span of the other column vectors, it is called a \emph{coloop}. The associated positroid cell to this decorated permutation is defined as

\[S_{\pi} = \{C \in \text{Gr}(k,n)^{\geq 0} | \pi_{C} = \pi \} \]
\end{definition}

Postnikov showed that $S_{\pi}$ is a cell, and that the positive Grassmannian $Gr(k,n)^{\geq 0}$ is the union of cells $S_{\pi}$ where $\pi$ ranges over decorated permutations of type $(k, n)$ \cite{https://doi.org/10.48550/arxiv.math/0609764}.

\begin{definition}
A \emph{plabic graph} is an undirected planar graph $G$ drawn inside a disk (considered modulo homotopy) with $n$ boundary vertices on the boundary of the disk, labeled $1, \hdots , n$ in clockwise order, as well as some internal vertices. Each boundary vertex is incident to a single edge, and each internal vertex is colored either black or white. If a boundary vertex is incident to a leaf (a vertex of degree 1), it is called a lollipop.    
\end{definition}

\begin{definition}
A \emph{perfect orientation} $\mathcal{O}$ of a plabic graph $G$ is a choice of orientation of each of its edges such that each black internal vertex $u$ is incident to exactly one edge
directed away from $u$; and each white internal vertex $v$ is incident to exactly one edge directed toward $v$. A plabic graph is called \emph{perfectly orientable} if it admits a perfect
orientation. Let $G_{\mathcal{O}}$ denote the directed graph associated with a perfect orientation $\mathcal{O}$ of
$G$. The \emph{source set} $I_{\mathcal{O}} \subseteq [n]$ of a perfect orientation $\mathcal{O}$ is the set of $i$ which are sources of the directed graph $G_{\mathcal{O}}$. Similarly, if $j \in \overline{I_{\mathcal{O}}} := [n] - I_{\mathcal{O}}$, then $j$ is a \emph{sink} of $\mathcal{O}$.   
\end{definition}

The following result links positroids with plabic graphs \cite{https://doi.org/10.48550/arxiv.math/0609764,m=2amplut}.

\begin{theorem}[Theorem 12.6 \cite{m=2amplut}]
Let $G$ be a plabic graph of type $(k, n)$. Then we have a
positroid $M_{G}$ on $[n]$ defined by

\[ M_{G} = \bigg\{ I_{\mathcal{O}} \> | \> \mathcal{O} \text{ \> is \> a \> perfect \> orientation \> of \>} G \bigg\} \]

where $I_{\mathcal{O}}$ is the set of sources of $\mathcal{O}$. Moreover, every positroid cell has the form $S _{M_{G}}$ for some plabic graph $G$.
\end{theorem}

If a plabic graph $G$ is \emph{reduced} \cite{https://doi.org/10.48550/arxiv.math/0609764,fomin2021introduction} we have that $S_{M_{G}} = S_{\pi_{G}}$ , where $\pi_{G}$ is the decorated permutation defined as follows.

\begin{definition}
Let $G$ be a reduced plabic graph with boundary vertices $1, \hdots , n$. For each boundary vertex $i \in [n]$, we follow a path along the edges of $G$ starting at $i$, turning
(maximally) right at every internal black vertex, and (maximally) left at every internal white vertex. This path ends at some boundary vertex $\pi(i)$. The fact that $G$ is reduced implies that each fixed point of $\pi$ is attached to a lollipop; we color
each fixed point by the color of its lollipop. This defines a decorated permutation, called the \emph{decorated trip permutation} $\pi_{G} = \pi$ of $G$.    
\end{definition}

In \cite{oh2011positroids}, the following result elaborates on the way to compute the associated decorated permutations of an LPM.

\begin{theorem}[Theorem 25 \cite{oh2011positroids}]
Let I and J be two lattice paths starting at the origin and terminating at $(k,n-k)$, such that $I$ never crosses $J$. Let $I = \{i_{1} < \hdots < i_{k} \}$ and $J = \{ j_{1} < \hdots < j_{k} \} \in \binom{[n]}{k}$. Denote $[n] \setminus J =
\{d_{1} < \hdots < d_{n - k} \}$ and $[n] \setminus I = \{ c_{1} < \hdots < c_{n - k} \}$ . Then $\M[I,J]$ is a positroid and its decorated permutation $\pi_{\M[I,J]}$ is given by: 

\begin{align*}
  \pi_{\mathcal{M}[I,J]}(j_{r}) =  i_{r} \quad \forall r \in [k] \\
  \pi_{\mathcal{M}[I,J]}(d_{r}) =  c_{r} \quad \forall r \in [n-k] 
\end{align*}

if $\pi_{\mathcal{M}[I,J]}(t) = t$, then,

\[ \text{col}(t) = \left\{ 
  \begin{aligned}
  & -1 \quad \text{if} \>\>  t \in J \\ 
  &  1 \quad \quad \text{otherwise}
  \end{aligned}
\right\} \]
\end{theorem}

where ${col}()$ represents the coloring map for loop and co-loop elements of the permutation. Figure \ref{fig:LPM_(3,6)_dec_perm} lists the decorated permutations and the reduced plabic graphs corresponding to the snakes in the snake decomposition of $\U_{3,6}$ described in Figure \ref{fig:LPM_(3,6)}.

\begin{figure}
    \centering
    \includegraphics[scale=0.2]{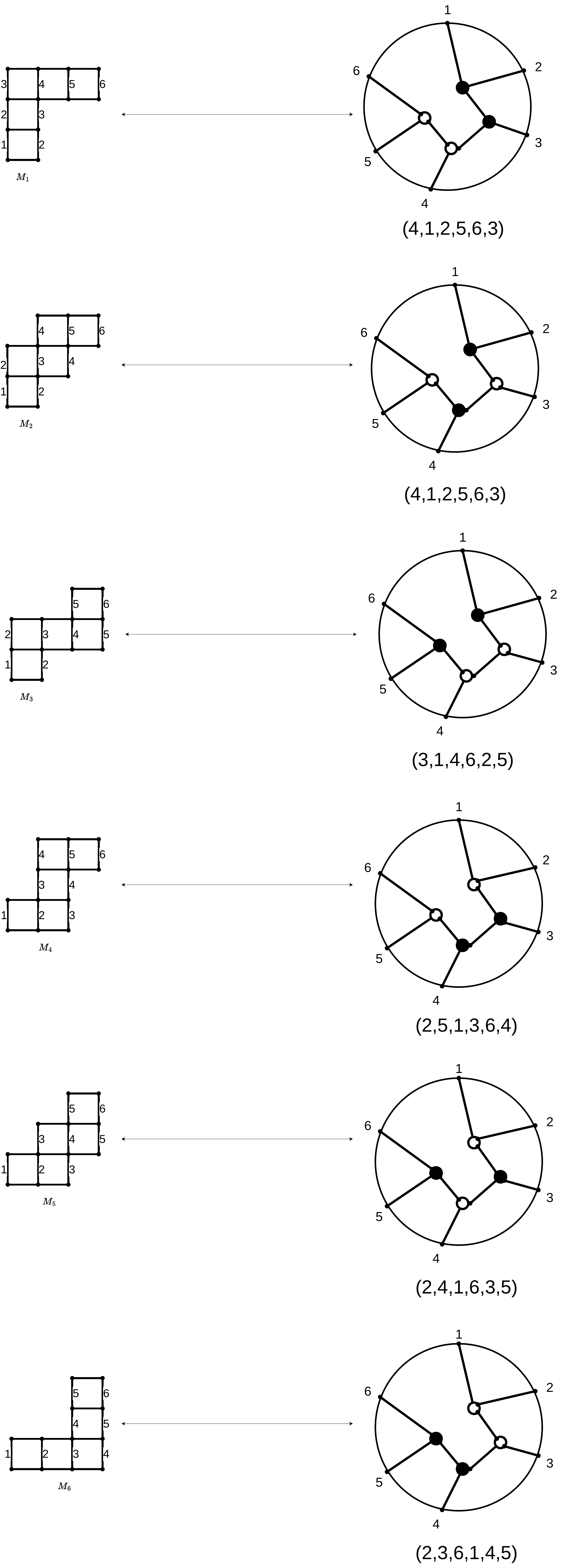}
    \caption{Decorated permutations and reduced plabic graphs corresponding to snakes of $\U_{3,6}$}
    \label{fig:LPM_(3,6)_dec_perm}
\end{figure}

\subsection{LPMfan(3,6)}

We first inspect the \emph{$f$-vector} of the fans associated to $\U_{3,6}$ \cite{herrmann2009draw}

\[  \text{$f$-vector} (\Dr(3,6)) = (1,65,535,1350,1005)  \]
\[  \text{$f$-vector} (\Trop(\Gr(3,6))) = (1,65,550,1395,1035)
  \]
 Out of the 65 rays of the Dressian $Dr(3,6)$, 35 correspond to splits and lie in the split complex, whereas the other 30 correspond to coarsest subdivisions of $\Delta(3,6)$ into three maximal cells.

Restricting to the positive tropical Grassmannian, we get the following vector \cite{arkani2021positive}, \cite{speyer2005tropical}, \cite{m=2amplut} where $F_{3,6}$ is the fan associated to $\text{Trop}^{+}(Gr(3,6))$.

\[  \text{f-vector} (F_{3,6}) = (1,16, 66, 98, 48)
  \]

Out of these 16 rays, five occur in the LPMfan in the form of $S_{1},S_{2},S_{3},S_{4}$ and $S_{5}$ which we see in Figure \ref{fig:LPM_fan}. The \emph{f-vector} of the LPMfan for $\Delta(3,6)$ is listed below where all cones are obtained as refinements of the five splits $S_{1},S_{2},S_{3},S_{4}$ and $S_{5}$, illustrated in Figure \ref{fig:LPM_fan}, where the edges between cones labeled signify the combination of the corresponding splits.

\[  \text{f-vector} (\text{LPMfan}(\Delta(3,6)) = (1,5,7,3,1)
  \]

The LPMfan(3,6) sits inside the Split subcomplex generated by the refinements of splits $S_{1},S_{2},S_{3},S_{4}$ and $S_{5}$. Also, to reiterate the cones are defined as secondary cones with rays defined by the corresponding splits, i.e., the collection of all the weight vectors which induce the same LPM subdivision lie in the same cone.

\begin{definition}\label{def:split_snake}
We refer to a lattice path matroidal subdivision which is a split with a snake as a maximal cell as \emph{snake split subdivision} and we refer to the snakes appearing in a snake split subdivision  as \emph{split snakes}.    
\end{definition}

\begin{figure}
    \centering
    \includegraphics[width=\textwidth,height=14cm, keepaspectratio]{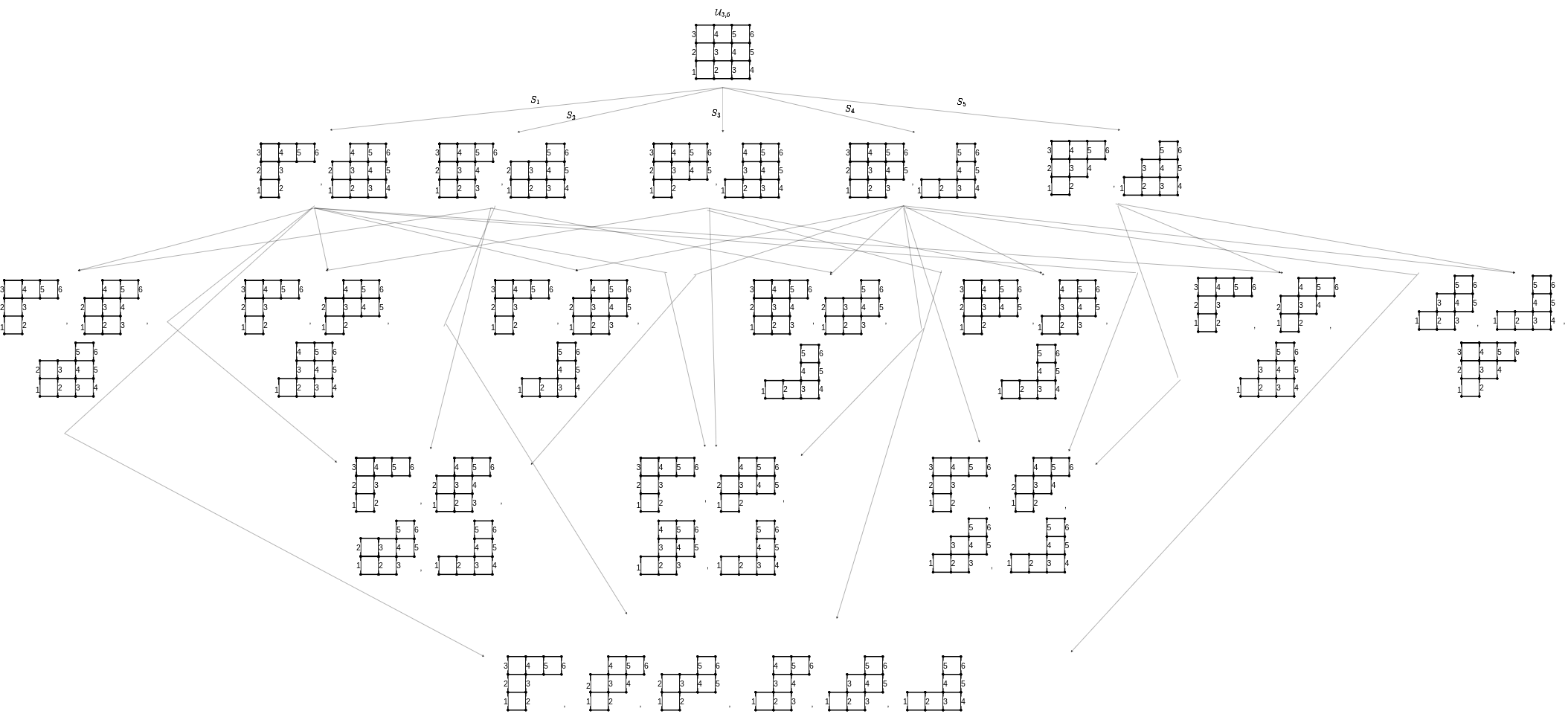}
    \caption{LPMfan(3,6).}
    \label{fig:LPM_fan}
\end{figure}

\begin{remark}
We point out the LPM decompositions for $\mathcal{U}_{(3,6)}$ other than the ones shown in Figure \ref{fig:LPM_(3,6)_decomp}, and these are depicted in Figure \ref{fig:SplitS'}, and one of the weight vectors inducing the split subdivision $S_{5}$ is the zero vector. 

\begin{figure}
    \centering
    \includegraphics[scale=0.23]{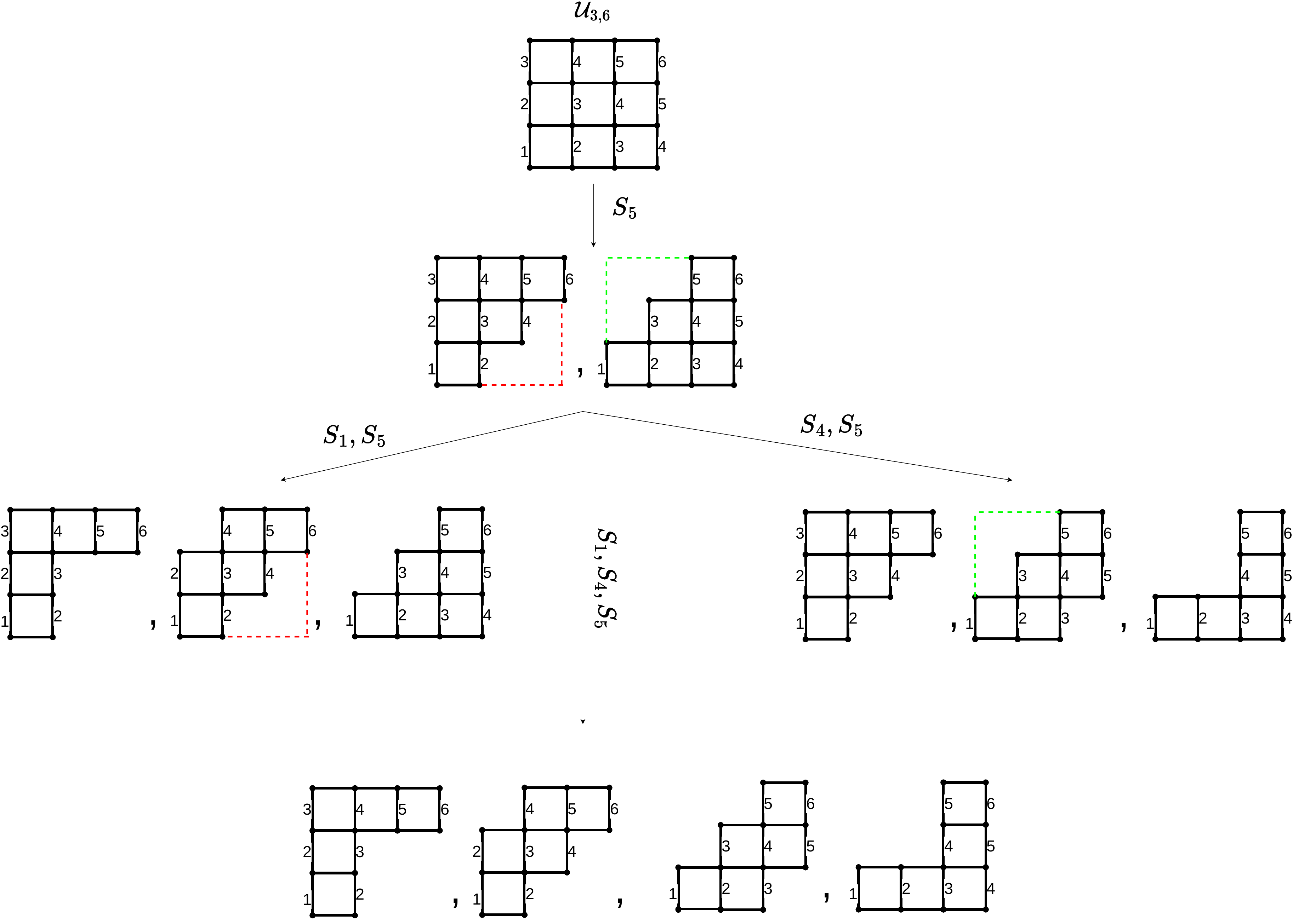}
    \caption{A split LPM decomposition $S'$ along with other decompositions that it corresponds to.}
    \label{fig:SplitS'}
\end{figure}
  
\end{remark}

\[ w_{S'} = \{0,0,0,0,0,0,0,0,0,0,0,0,0,0,0,0,0,0,0,0\} \]

\begin{remark}
We also want to point out to the reader that there exists a natural action of the symmetric group $S_{n}$ on the cones of the Dressian $Dr(\Delta(k,n))$ well documented in \cite{herrmann2009draw} and well described in their computations, and that is why with respect to this action there are only 7 maximal cells of $\Dr(3,6)$ \cite{herrmann2009draw}. Our description of the LPMfan implicitly incorporates this symmetry, for example, the weight vectors 

 \[w_{1} = \{0, 0, 0, 0, 0, 0, 0, 0, 0, 0, 0, 0, 0, 0, 0, 0, 1, 1, 1, 1\}\]

 and 

 \[ w_{2} = \{1, 1, 1, 1, 0, 0, 0, 0, 0, 0, 0, 0, 0, 0, 0, 0, 0, 0, 0, 0\} \]

 both induce the split $S_{3}$, but we know that both of them are equivalent under the action of $S_{6}$.
\end{remark}

\subsection{Computations for LPM polytope $\Delta(4,8)$}

\begin{figure}
    \centering
    \includegraphics[scale=0.2]{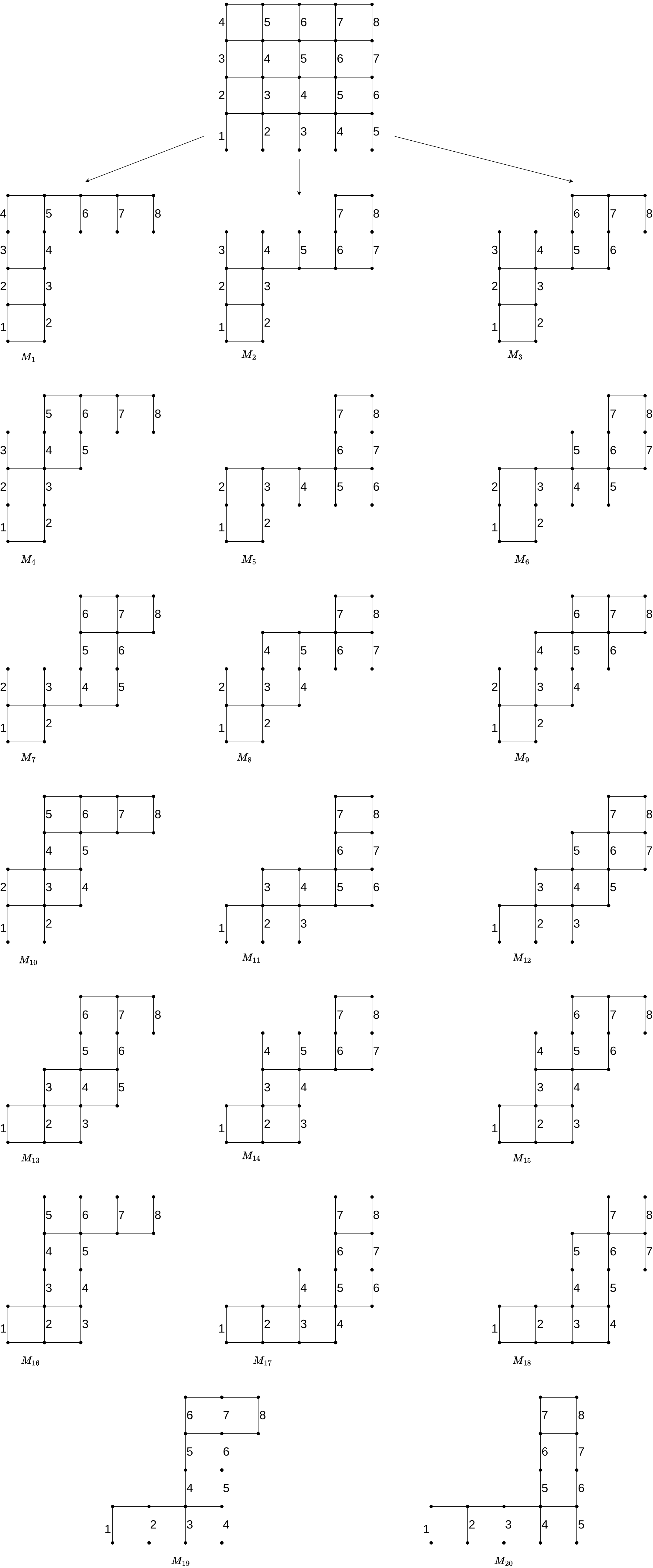}
    \caption{A LPM decomposition of $\U_{4,8}$ which corresponds to the LPM subdivision in Figure \ref{fig:LPM_(4,8)_subdiv} and indexes a maximal cone of $\Dr^{+}(4,8)$.}
    \label{fig:LPM_(4,8)}
\end{figure}

The subdivision is described in Figure \ref{fig:LPM_(4,8)} and Figure \ref{fig:LPM_(4,8)_subdiv} and this subdivision is induced by the weight vector

\begin{figure}
    \centering
    \includegraphics[scale=0.345]{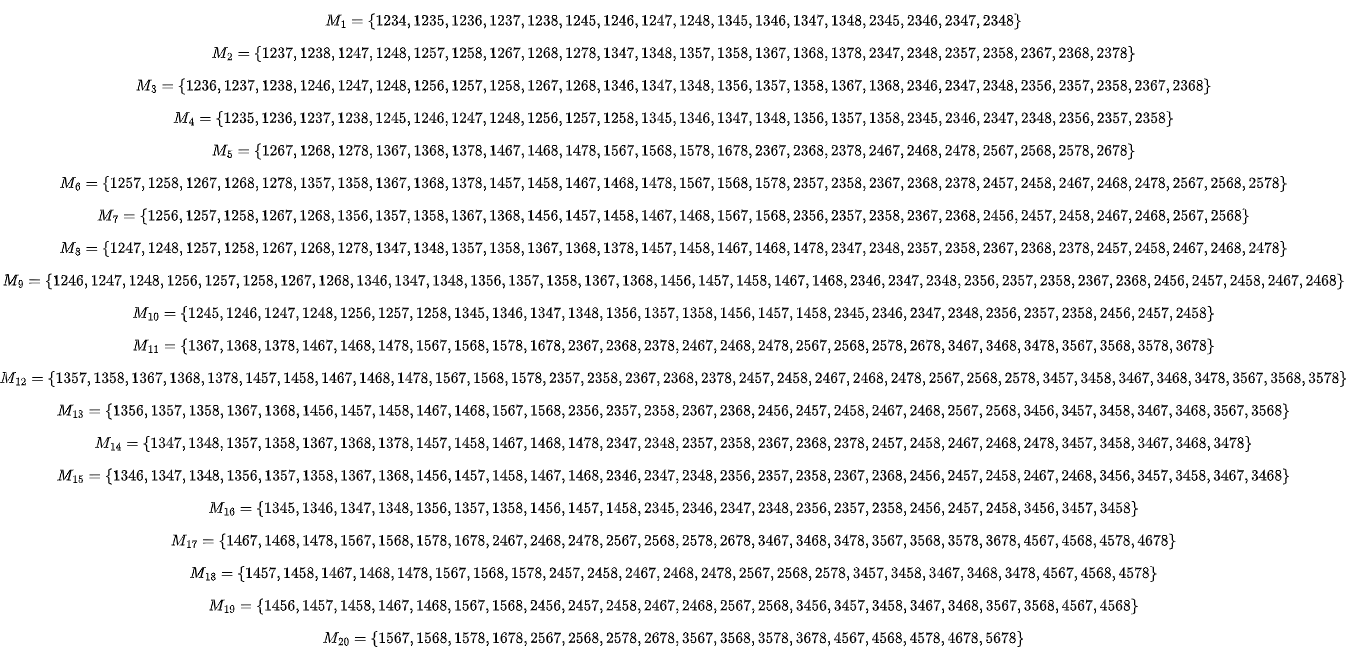}
    \caption{The LPM subdivision of $\Delta(4,8)$, corresponding to the decomposition in Figure \ref{fig:LPM_(4,8)}.}
    \label{fig:LPM_(4,8)_subdiv}
\end{figure}

\begin{multline*} 
  w = \{ 0,0,0,0,0,0,0,0,0,1,1,1,2,2,3,0,0,0,0,1,1,1,2,2,3,2, \\ 2,2,3,3,4,5,5,6,8,0,0,0,0,1,1,1,2,2,3,2,2,2,3, \\
3,4,5,5,6,8,3,3,3,4,4,5,6,6,7,9,8,8,9,11,14 \}   
\end{multline*}

A subsequent computation for the LPMfan$(\Delta(4,8))$ is more intricate than the computation of LPMfan$(\Delta(3,6))$ and hence we leave it for future work as we believe it would be nice to utilize the symmetric group action also into the computation to produce bigger examples.

All the files containing the code used for all these computations can be found at the following link

\begin{center}
    \url{https://github.com/Ayush-Tewari13/LPM_SUBDIVISIONS}
\end{center}

\section{Amplituhedron and Positive configuration spaces}\label{sec:Sec6}

We now describe an important implication of our results and connections to topics in Physics, which in recent times have gained immense interest. In \cite{arkani2014amplituhedron}, Arkani-Hamed et. al introduced the notion of the \emph{amplituhedron} which is obtained from the positive Grassmannian via the \emph{amplituhedron map}. It has been noted that the amplituhedron encodes information concerning scattering amplitudes in $\mathcal{N}=4$ super Yang-Mills theory, which in turn explains the etymology of the term. In \cite{m=2amplut}, the authors introduce the notion of \emph{positroid dissections} for the hypersimplex $\Delta(k+1,n)$ and the \emph{Grasstopes dissection} for the amplituhedron and explain the ways in these two dissections can be related via a duality map. 

We begin with the definition of amplituhedron \cite{arkani2014amplituhedron}, \cite{m=2amplut},

\begin{definition}
For $a \leq b$, define $\text{Mat}^{>0}_{a,b}$ as the set of real $a \times b$ matrices whose $a \times a$ minors are all positive. Let $Z \in \text{Mat}^{>0}_{n,k+m}$. The amplituhedron map $\overline{Z} : \text{Gr}(k,n)^{\geq 0} \rightarrow \text{Gr}(k,k+m)$ is defined by $\overline{Z} := CZ$, where $C$ is a $k \times n$ matrix representing an element of $\text{Gr}(k,n)^{\geq 0}$ and $CZ$ is a $k \times (k + m)$ matrix representing an element of $\text{Gr}(k,k+m)$ . The \emph{amplituhedron} $\mathcal{A}^{\geq 0}_{n,k,m}(Z) \subseteq \text{Gr}(k,k+m)$ is the image $\overline{Z}(\text{Gr}(k,n)^{\geq 0})$.    
\end{definition}

We briefly state some of the results from \cite{m=2amplut} to sketch the outline of their discussion,

\begin{definition}\label{def:pos_dissetion}
Let $\mathcal{C} = \{ \Gamma_{\pi}\}$ be a collection of positroid polytopes, and let $S_{\pi}$ be the collection of corresponding positroid cells. $\mathcal{C}$ is a \emph{positroid dissection} of $\Delta(k,n)$ if 

\begin{itemize}
    \item  dim$(\Gamma_{\pi}) = n-1$ for each $\Gamma_{\pi} \in \mathcal{C}$
    \item pairs of two distinct positroid polytopes $\Gamma^{o}_{\pi} = \mu(S_{\pi})$ and $\Gamma^{o}_{\pi'} = \mu(S_{\pi'})$ are pairwise disjoint, and 
    \item $\cup_{\pi} \Gamma = \Delta(k,n)$.
\end{itemize}
\end{definition}
    
\begin{definition}
Let $A$ be a $k \times n$ matrix representing a point in $Gr(k,n)^{\geq 0}$. The \emph{moment map} $\mu: Gr(k,n)^{\geq0} \rightarrow \mathbb{R}^{n}$ is defined by

\[ \mu(A) = \frac{\sum_{I \in \binom{[n]}{k}}|p_{I}(A)|^{2}e_{I}}{ \sum_{I \in \binom{[n]}{k}}|p_{I}(A)|^{2}}\]

\end{definition}

A positroid dissection is called a \emph{positroid tiling} if $\mu$ is injective on each $S_{\pi}$. As can be seen from the definition, dissections are a more generalized notion of a polytopal subdivision for a hypersimplex, with no restrictions on how individual pieces meet at the boundary, although the notion of \emph{good dissections} \cite{m=2amplut} exactly agrees with the notion of a subdivision,

\begin{definition}
Let $\mathcal{C} = \{\Gamma_{\pi^{(1)}} , \hdots , \Gamma_{\pi^{(l)}}\}$ be a dissection of $\Delta(k+1,n)$ . We say that $\mathcal{C}$ is a good dissection of $\Delta(k+1,n)$ if the following condition is satisfied: for $i \ne j$, if $\Gamma_{\pi^{(i)}} \cap \Gamma_{\pi^{(j)}}$ has
codimension one, then $\Gamma_{\pi^{(i)}} \cap \Gamma_{\pi^{(j)}}$ equals $\Gamma_{\pi}$, where $\Gamma_{\pi}$ is a facet of both $\Gamma_{\pi^{(i)}}$ and $\Gamma_{\pi^{(j)}}$.    
\end{definition}

In \cite{m=2amplut} a dissection for the hypersimplex is provided inspired by \emph{BCFW recurrence} relations for tilings of the $m=4$ amplituhedron, which is referred as the \emph{BCFW-style recurrence}

\begin{theorem}[Theorem 4.5 \cite{m=2amplut}]\label{thm:BCFW_rec_m=2}
Let $\mathcal{C}_{k+1,n-1}$ (respectively $\mathcal{C}_{k,n-1})$ be a collection of positroid polytopes that dissects the hypersimplex $\Delta(k+1,n-1)$ (respectively $\Delta(k,n-1))$. Then 

\[ \mathcal{C}_{k+1,n}  = i_{\text{pre}} (\mathcal{C}_{k+1,n-1}) \cup  i_{\text{inc}}(\mathcal{C}_{k,n-1}) \]

dissects $\Delta(k+1,n)$,where $i_{\text{pre}}$ and $i_{\text{inc}}$ are maps defined on reduced plabic graphs in \cite[Definition 4.1]{m=2amplut}.

\end{theorem}

\subsection{Matroidal definition for BCFW dissections of hypersimplex}

We now try to build a purely matroidal relation for BCFW-style recurrence dissection for hypersimplices.


    

We provide some context to our notations. For a positroid polytope $\mathcal{P}$, we refer to the underlying positroid as $\mathcal{P} = \conv(\mathcal{M})$, where $\conv$ represents taking the convex hull of the indicator vectors of the bases of $\mathcal{M}$. We now provide the matroidal definition for \emph{BCFW style recurrence dissections} for the hypersimplex.

\begin{theorem}[matroidal BCFW style relations for hypersimplex]\label{thm:BCFW_matroidal}

Let $\mathcal{C}_{k+1,n}$ be a collection of positroid polytopes that dissects the 
hypersimplex $ \Delta(k+1,n) = \conv(\mathcal{U}_{k+1,n})$. 

Then, 

\[\mathcal{C}_{k+1,n} =  ((\mathcal{C}_{k+1,n}) / e_{i}) \cup ((\mathcal{C}_{k+1,n}) \setminus e_{i})      \]

and the set $(\mathcal{C}_{k+1,n}) / e_{i})$ provides a positroid dissection of $\Delta(k,n-1)$ and $(\mathcal{C}_{k+1,n}) \setminus e_{i})$ provides a positroid dissection of $\Delta(k+1,n-1)$ , where $'/'$ represents contraction and $'\setminus'$ represents the deletion operations on matroids.
\end{theorem}

\begin{proof}
Firstly we note that the hypersimplex $\Delta(k+1,n)$, is a 0-1 polytope obtained by the intersection of the unit cube $[0,1]^{n}$ with the affine hyperplane $\sum_{i=1}^{n} x_{i} = k+1$. We note that the facet corresponding to the hyperplane $x_{i}=0$ is termed as the \emph{i-th deletion facet} of $\Delta(k+1,n)$ and is isomorphic to $\Delta(k+1,n-1)$. Similarly, the facet corresponding to the hyperplane $x_{i}=1$ is termed as the \emph{i-th contraction facet} of $\Delta(k+1,n)$ and is isomorphic to $\Delta(k,n-1)$. Also, these facets can be obtained as deletion and contraction respectively on the uniform matroid $\mathcal{U}_{k,n}$ \cite{herrmann2009draw}. 

With these definitions, the notions of contraction and deletion extend to respective dissections and subdivisions and this fact is used in \cite{herrmann2009draw}. We point out the natural dissections of hypersimplex into two minors provided by contraction and deletion. Let $v \in \text{Vert}(\Delta(k+1,n))$ then since each dissection into $\conv \>((\mathcal{M}_{k+1,n}) / e_{i})$ and $\conv((\mathcal{M}_{k+1,n}) \setminus e_{i})$ is defined by hyperplanes $x_{i} = 0$ or $x_{i}=1$, therefore every vertex $v$ lies in either $\conv \>((\mathcal{M}_{k+1,n}) / e_{i})$ or $\conv((\mathcal{M}_{k+1,n}) \setminus e_{i})$. 

Given a positroid dissection $\mathcal{C}_{k+1,n-1}$ we consider the minors with respect to an element $i \in [n]$ and obtain minors $((\mathcal{C}_{k+1,n}) / e_{i})$ and $((\mathcal{C}_{k+1,n}) \setminus e_{i})$. We recognize that these minors also correspond to the dissections induced on the respective contraction and deletion facets of $\Delta(k+1,n)$ respectively, each of which is isomorphic to $\Delta(k,n-1)$ and $\Delta(k+1,n-1)$ respectively, which give us the two required positroid dissections. 
\end{proof}

\begin{remark}
We point out that Theorem \ref{thm:BCFW_matroidal} provides a  matroidal formulation of BCFW style relations for hypersimplex, and proves an almost converse statement of Theorem \ref{thm:BCFW_rec_m=2}, and we say that this is almost a converse statement since we know that not all positroid dissections of $\Delta(k+1,n)$ occur from BCFW style recursions \cite{m=2amplut}, whereas the statement of Theorem \ref{thm:BCFW_matroidal} involves matroidal operations, therefore for any positroid dissection of $\Delta(k+1,n)$ we can obtain dissections of $\Delta(k,n-1)$ and $\Delta(k+1,n-1)$ in this way. We point out that it is not obvious that there exist matroidal operations equivalent to the operations $i_{pre}$ and $i_{inc}$ used in Theorem \ref{thm:BCFW_rec_m=2}. We also wish to explore a possible generalization of the statement for Theorem \ref{thm:BCFW_matroidal} to matroid dissections and not necessarily positroid dissections. However, such a discussion would require an appropriate definition of a matroid dissection and a  generalization of Theorem \ref{thm:BCFW_rec_m=2} to the case of matroid dissections, as the non-trivial part of the proof of Theorem \ref{thm:BCFW_rec_m=2}  rests on a refined description of facets of positroid polytopes defined by Postnikov, described in \cite[Proposition 5.6]{ardila2016positroids}.  
\end{remark}

\begin{example}
We again consider the snake polytope decomposition of $\Delta(3,6)$ described in Figure \ref{fig:LPM_(3,6)}. As we know that this is also a regular positroidal subdivision, or equivalently a regular positroid good dissection. We now perform the contraction and deletion with respect to the element $i = 1$ on this subdivision, and obtain two collections in which we see that $\{ M_{2} \setminus \{1\}, \hdots , M_{6} \setminus \{1\}\}$ provides a positroidal subdivision (equivalently a positroid good dissection) of $\Delta(3,5)$ on letters $[6] \setminus \{1\} = \{2,3,4,5,6\}$ and $\{ M_{1} / \{1\}, M_{2} / \{1\}, M_{3} / \{1\}\}$ provides a positroidal subdivision on $\Delta(2,5)$ (cf. from Figure \ref{fig:contr_del}.)

\begin{figure}
    \centering
    \includegraphics[scale=0.45]{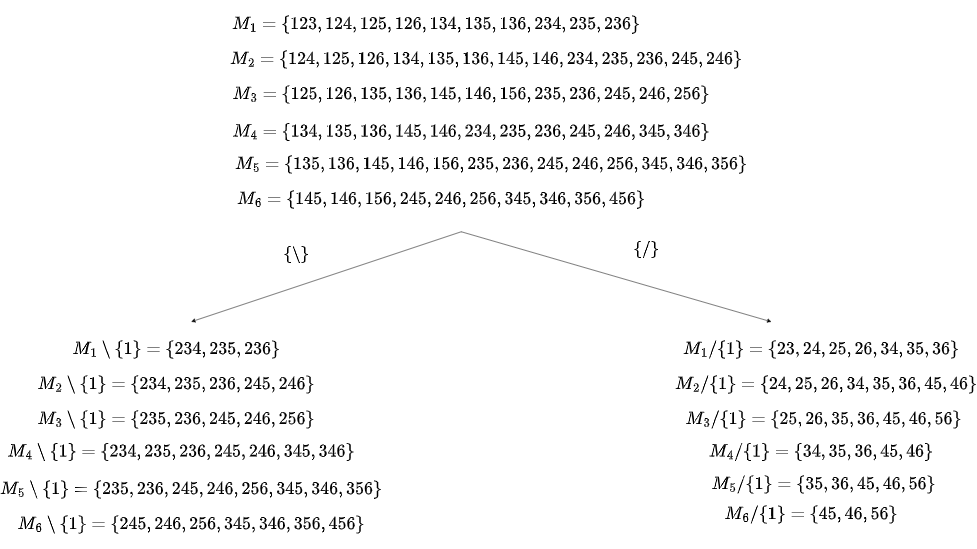}
    \caption{The contraction and deletion operation on the positroidal subdivision of $\Delta(3,6)$ from Figure \ref{fig:LPM_(3,6)}. }
    \label{fig:contr_del}
\end{figure}

\end{example}

\subsection{BCFW cells correspond to lattice path matroids}

We want to point out that in the discussion in this section, we would only be focusing on the m=4 amplituhedron.

In a recent breakthrough work \cite{even2021amplituhedron}, the authors prove the conjecture that BCFW cells provide a triangulation of the amplituhedron $\mathcal{A}_{n,k,4}$. In \cite{even2021amplituhedron} and \cite{karp2020decompositions} the authors establish the equivalence between BCFW cells and noncrossing lattice walks (paths). We use this observation to explore the connection between BCFW triangulations and lattice path matroids.

We borrow mostly our notation from \cite{karp2020decompositions}. Let $\mathcal{L}_{n,k,4}$ denote the set of all pairs $(P_{\mathcal{L}}, Q_{\mathcal{L}})$ of \emph{noncrossing} lattice paths inside a $k \times (n - k - 4)$ rectangle, where the notion of noncrossing is the same as $P$ never going above $Q$ implicit in Definition \ref{def:LPM}. Therefore, we state one of our first conclusions in the form of Corollary \ref{cor:noncrossing}.

\begin{corollary}\label{cor:noncrossing}
Let $(P_{\mathcal{L}}, Q_{\mathcal{L}})  \in \mathcal{L}_{n,k,4}$ be a pair of noncrossing lattice paths. Then $(P_{\mathcal{L}}, Q_{\mathcal{L}})$ determine a lattice path matroid $\mathcal{M}[P_{\mathcal{L}}, Q_{\mathcal{L}}]$ which lies inside the lattice path matroid $\mathcal{U}_{k,n-4}$. 
\end{corollary}

We describe the connection between non-crossing lattice paths and BCFW cells of $\mathcal{A}_{k,n,4}$. Firstly, in\cite{karp2020decompositions} the authors introduce the notion of a $\oplus-\text{diagram}$ of type $(k,n)$, which are defined as follows \cite[Definition 2.3]{karp2020decompositions}

\begin{definition}
Fix $0 \leq k \leq n$. Given a partition $\lambda$, we let $Y_{\lambda}$ denote the Young diagram of $\lambda$. A $\oplus-\text{diagram}$ of type $(k,n)$ is a filling $D$ of a Young diagram $Y_{\lambda}$ fitting inside a $k \times (n - k)$ rectangle with the symbols $0$ and $+$ (such that each box of Y is filled with exactly one symbol) and  $\lambda$ is called the shape of $D$ (cf. Figure \ref{fig:plus_diag}).  
\end{definition}

\begin{figure}
    \centering
    \includegraphics[scale=0.6]{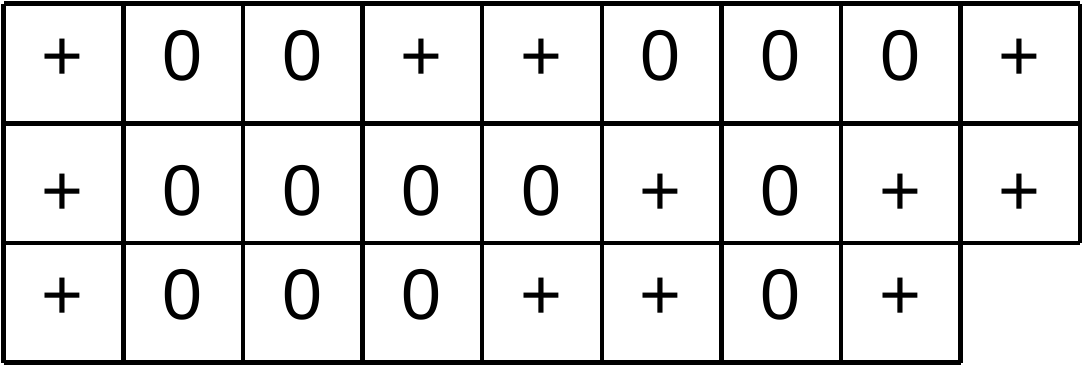}
    \caption{A $\oplus-$diagram $D$ of type $(3, 12)$.}
    \label{fig:plus_diag}
\end{figure}

The rules according to which the filling in a $\oplus$-diagram is obtained are elaborated in \cite[Definition 6.2]{karp2020decompositions}. Let $\mathcal{D}_{n,k,4}$ be the space of $\oplus-\text{diagram}$ of type $(k,n)$. We infer the following result from \cite[Definition 6.2]{karp2020decompositions}

\begin{lemma}
There exists an bijection $\Omega_{\mathcal{L}\mathcal{D}}$ such that

\[  \Omega_{\mathcal{L}\mathcal{D}} : \mathcal{L}_{n,k,4} \rightarrow \mathcal{D}_{n,k,4} \]
\end{lemma}

\begin{theorem}[Theorem 6.3 \cite{karp2020decompositions}]
The $\oplus-\text{diagrams}$ $\mathcal{D}_{n,k,4}$ index the $(k, n)$-BCFW cells $\mathcal{C}_{n,k,4}$ .    
\end{theorem}

This theorem is proven by using another bijection between the space of \emph{binary rooted trees} $\mathcal{T}_{n,k,4}$ and $\mathcal{L}_{n,k,4}$ and the authors use reduced plabic graphs to produce \emph{decorated permutations} for the $\oplus-\text{diagrams}$. We point the reader to \cite{karp2020decompositions} to explore these concepts and proofs in full detail. Our interest develops with Corollary \ref{cor:noncrossing} and this inspires us to enquire about the existence of a duality between cells of the amplituhedron and dissections of the hypersimplex, which is established via T-duality in the case of $m=2$ amplituhedron in \cite{m=2amplut}. In \cite{even2021amplituhedron} the following result concerning BCFW cells is proven, which was stated as a conjecture in \cite{arkani2014amplituhedron,karp2020decompositions}.

\begin{theorem}
For every $k \geq 1$ and $n \geq k+4$, the $(k, n)$-BCFW cells
form a triangulation of the amplituhedron $\mathcal{A}_{n,k,4}$.    
\end{theorem}

We now state our result based on this discussion,

\begin{theorem}\label{thm:m=4_aplitu_BCFW_LPM}
Each triangulation of the amplituhedron $\mathcal{A}_{n,k,4}$ into $(k, n)$-BCFW cells provides a positroid dissection $\{\Gamma_{i}\}$ of the hypersimplex $\Delta(k,n-4)$, where each BCFW cell corresponds to a lattice path matroid polytope $\Gamma_{i}$.     
\end{theorem}

\begin{proof}
By Corollary \ref{cor:noncrossing} we already know that each $(k,n)$ BCFW cell corresponds to a LPM $\mathcal{M}[P_{\mathcal{L}},Q_{\mathcal{L}}]$ inside $\mathcal{U}_{k+4,n}$, where $(P_{\mathcal{L}},Q_{\mathcal{L}}) \in \mathcal{L}_{n,k,4}$. Therefore, each $(k,n)$ BCFW cell corresponds to a lattice path matroid polytope $\mathcal{P}(\mathcal{M}[P_{\mathcal{L}},Q_{\mathcal{L}}])$ which lies inside $\Delta(k+4,n) = \conv(\mathcal{U}_{k+4,n})$. Therefore, a triangulation of $\mathcal{A}_{n,k,4}$ into $(k, n)$-BCFW cells corresponds to a collection of all lattice path matroids which lie inside the uniform matroid $\Delta(k+4,n) = \conv(\mathcal{U}_{k+4,n})$, which is clearly a positroid dissection from Definition \ref{def:pos_dissetion}.     
\end{proof}

With Theorem \ref{thm:m=4_aplitu_BCFW_LPM} we establish a first notion in the direction of \emph{T-duality} for a $m=4$ amplituhedron, where in the case of $m=2$ amplituhedron \cite{m=2amplut} shows that subdivisions of the amplituhedron correspond to positroid dissections of the corresponding hypersimplex. We provide this in the case of $m=4$ amplituhedron for the BCFW triangulation which inspires for the exploration of the case of other triangulations and subdivision of $\mathcal{A}_{k,n,4}$. Also, BCFW style dissections enjoy a recursive description and can be understood as coming from splits as discussed in the case of the $m=2$ amplituhedron in \cite[Remark 4.8]{m=2amplut}, and we believe that a positroid dissection in LPM cells captures this in essence as well owing to the recursive definition of LPM polytope decompositions.

\subsection{Positive configuration spaces, weakly separated collections and connected minimal positroids}\label{subsec:pos_conf_space}

We highlight some of the connections between our study on LPM's and \cite{arkani2021positive}. Firstly, in \cite{arkani2021positive} the authors relate the \emph{positive Chow cells} of the \emph{Chow quotient} of the Grassmannian with positroidal subdivisions. Let $Ch(k,n)_{\geq 0}$ denote the nonnegative part of the Chow quotient of the Grassmannian.

\begin{theorem}[Theorem 1.1 \cite{arkani2021positive}]
There are canonical bijections between the following sets.   
\begin{itemize}
    \item The set $\{\Theta_{\Tilde{\Delta} > 0} \}$ of positive Chow cells of $Ch(k,n)_{\geq 0}$
    \item The set $D(k,n)$ of regular positroidal subdivisions of $\Delta(k,n)$.
    \item The set of cones in the positive tropical Grassmannian $\text{Trop} \> Gr^{+}(k, n)$, the space of valuations of positive Puiseux series points $Gr(k, n)(\mathcal{R} > 0)$
    \item The set of cones in the positive Dressian $Dr(k,n)$, which satisfy the three term positive Pl\"ucker relations.
\end{itemize}
\end{theorem}

As LPM'S are positroids too, all these equivalences also are true when restricted to the LPMfan. 

We also delve into the connection between \emph{cluster} of a matroid, \emph{weakly separated collections} \cite{oh2015weak} and snakes. We fix some notations relevant to our discussion. We
deﬁne the \emph{cyclic ordering} \cite{oh2011positroids} (referred as the \emph{$t$-th Gale order} in \cite{benedetti2023lattice}) $\leq _{t}$ on $[n]$ for some $t \in [n]$ by the total order $t \leq_{t} t + 1 \leq_{t} \hdots \leq_{t} n \leq_{t} 1 \hdots \leq_{t} t - 1$. For $I ,J \in \binom{[n]}{k}$, where

\[ I = \{ i_{1} , \hdots , i_{k} \}, \quad i_{1} \leq_{t} i_{2} \hdots \leq_{t} i_{k}  \]

and 

\[ J = \{ j_{1} , \hdots , j_{k} \}, \quad j_{1} \leq_{t} j_{2} \hdots \leq_{t} j_{k}   \]

then 

\[ I \leq_{t} J \quad \text{if \> and \> only \> if\> } i_{1} \leq_{t} j_{1} , \hdots , i_{k} \leq_{t}  j_{k}  \]

\begin{definition}
For each $I \in  \binom{[n]}{k}$ and $t \in [n]$ , we deﬁne the \emph{cyclically shifted Schubert} matroid as

\[ SM_{I}^{t} = \Biggl\{ J \in \binom{[n]}{k} \>  | \> I \leq_{t} J \Biggl\} \]
\end{definition}

We recall the definition for weakly separated sets from \cite{oh2015weak},

\begin{definition}
 Let $I$ and $J$ be two subsets of $[n]$. $I$ and $J$ are said to be weakly separated if either
\begin{itemize}
    \item $|I| \leq |J|$ and $I \setminus J$ can be partitioned as $I_{1} \cup I_{2}$ such that $I_{1} \prec J \setminus I \prec I_{2}$ or
    \item $|J| \leq |I|$ and $J \setminus I$ can be partitioned as $J_{1} \cup J_{2}$ such that $J_{1} \prec I \setminus J \prec J_{2}$
\end{itemize}

where $A \prec B$ indicates that every element of $A$ is less than every element of $B$.   
\end{definition}

Equivalently, the definition can be stated as the sets $I$ and $J$ $\in$ $\binom{[n]}{k}$ are said to be \emph{weakly separated} if we cannot find cyclically ordered elements $a,b,c,d$ such that $a,c \in I \setminus J$ and $b,d \in J \setminus I$ (also along with the symmetrical statement for $I$ and $J$ swapped).

We also recall the definition of \emph{Grassmann necklaces} \cite{https://doi.org/10.48550/arxiv.math/0609764,oh2015weak,oh2011positroids,postnikov2018positive}.

\begin{definition}
A Grassmann necklace is a sequence $I = (I_{1} , \hdots , I_{n})$ of subsets $I_{r} \subseteq [n]$ such that:

\begin{itemize}
    \item if $i \in I_{i}$ then $I_{i + 1} = (I_{i} \setminus \{i \}) \cup \{ j \}$ for some $j \in [n]$ ,
    \item if $i \not \in I_{i}$ then $I_{i + 1} = I_{i}$
\end{itemize}

The indices are taken modulo $n$. In particular, we have $| I_{1} | = \hdots = | I_{n} |$ .
\end{definition}

There exists a canonical bijection between positroids and Grasmann necklaces. We state the characterization of the \emph{cluster} of a matroid (Definition \ref{def:cluster}), in terms of weakly separated sets and Grassmann necklaces \cite{oh2015weak}

\begin{lemma}
A subset $\mathcal{C} \subseteq \mathcal{M}$ is a cluster if it is pairwise weakly separated, has size $\text{dim}(\mathcal{M})$ + 1, and contains the Grassmann necklace $\mathcal{I}$ of ${M}$. Any pairwise
weakly-separated subset of $\binom{[n]}{k}$ can be extended to a cluster.    
\end{lemma}

As one of the takeaways in \cite{arkani2021positive}, the authors state this result concerning minimal connected positroids and clusters for them 

\begin{lemma}\label{lem:min_pos}
A connected positroid $\mathcal{M}$ is minimal if and only if  the associated reduced plabic graph $G(C)$ is a tree, for
some cluster $\mathcal{C}$ of $\mathcal{M}$. In this case, $\mathcal{M}$ has a unique cluster $\mathcal{C} \subseteq \mathcal{M}$.    
\end{lemma}

We already know that for lattice path matroids, snakes are minimal matroids. Hence, by Lemma \ref{lem:min_pos}, we obtain a unique cluster in this case. We explain this with one of our running examples; the snake decomposition of $\mathcal{U}_{3,6}$ shown in Figure \ref{fig:LPM_(3,6)}. We obtain the cluster $\mathcal{C}_{1}$ for the snake $M_{1}$ 

\[ \mathcal{C}_{1}  = \{123,234,134,124,125,126\} \]

It is easy to verify that $C_{1}$ is weakly separated, contains the Grassmann necklace for $\mathcal{M}_{1}$ and is of cardinality = dim$(\mathcal{M}_{1}) +1 = 5+1 = 6$. Likewise, we obtain unique clusters for all the snakes. The corresponding graphs for these snakes have been described in Figure \ref{fig:LPM_(3,6)_dec_perm}.

We conclude with another interesting observation. For both the \emph{snake split} (Definition \ref{def:split_snake}) matroids, we notice that they exactly contain $k(n-k) +1$ elements. This is exactly equal to the cardinality of the \emph{maximal weakly separated collection}, which is the maximal collection of pairwise weakly separated elements inside the matroid $\mathcal{M}$ and the bound on its cardinality was famously conjectured by Leclerc and Zelevinsky and proven to be true in \cite{oh2015weak}. However, we realize that the elements in a snake split are not all pairwise weakly separated, so they are not examples of maximal weakly separated collections. For example for the snake decomposition of $\mathcal{U}_{3,6}$, the snake split $M_{1}$ has the elements $124$ and $135$ which are not pairwise weakly separated.

\section{Future Perspectives}\label{sec:Sec7}

We utilize this section to condense our discussion and to highlight the takeaways from our results. We also point to subsequent questions which arise from our work.

Firstly, we want to mention a recent work concerning lattice path matroid  decompositions into snakes and \emph{alcoved triangulation} \cite{benedetti2023lattice} in which the authors prove results based on the snake decomposition of LPM and also discuss results on Ehrhart theory of LPM. They prove that the alcoved triangulation of an alcoved polytope is regular. We observe that our discussion pertaining to lattice path matroidal subdivisions being regular generalizes this result for LPM. We point the reader to Figure 1 in \cite{benedetti2023lattice} to understand the context of where LPM lie with respect to other well-known families of matroids. 

We also want to point the reader to \cite{BrandtSpeyer}, where the authors show that there exist finest matroid subdivisions of matroid polytopes that do not contain matroid polytopes of indecomposable matroids as maximal cells. Hence, it might be a worthwhile question to ask for which other families of matroids apart from positroids,  a result like Corollary \ref{cor:indecompose} can be obtained, for example, the class of transversal matroids might be a good candidate to be considered. Also, a natural generalization of this question would be to consider the finest subdivisions of Dressians for arbitrary matroids, and not necessarily the hypersimplex, and see if we can recover some of these results.

With the introduction of the notion of LPMfan we believe that there are many questions that can be asked just pertaining to its structure and we believe this might interest readers for further research. Some of the interesting queries could be, to understand if there exists a bound on the number of LPM splits and how it behaves with respect to the Dressian, computation of the dimension of the LPMfan, etc. We believe there is much more to analyze about the LPMfan and we aim to fulfill this in future work.

We also acknowledge via \cite{benedetti2022lattice} recursive relations between LPM's defined by quotients and direct sums. This could be really interesting to understand LPM subdivisions for larger LPM polytopes. We wish to employ such a technique for computing LPMfan recursively. Additionally, it would be interesting to inquire about the specific Pl\"ucker relations that are satisfied by points corresponding to LPM subdivisions, which lie in the Dressian. We already know that they would be satisfying the positive Pl\"ucker relations owing to the fact that LPM are positroids, but they might be even further refined and this could be done by analyzing the forbidden minors for a matroid to be LPM, classified in \cite{bonin2010lattice}.

One of our future goals is also to find an equivalent of Theorem \ref{thm:BCFW_rec_m=2} for LPM dissections. Also, in \cite{m=2amplut}, the authors provide a characterization of the positroid polytope in the form of the following statement

\begin{theorem}[Theorem 3.9 \cite{m=2amplut}]
Let $M$ be a matroid of rank $k$, and consider the
matroid polytope  $P_{M}$. It is a positroid polytope if and only if all of its two-dimensional faces are positroid polytopes.    
\end{theorem}

We are able to obtain a one-way implication similar to this in the case of LPM polytopes as follows,

\begin{lemma}\label{lem:LPM_faces_are_lpm}
The faces of a lattice path matroid polytope $P_{M[P,Q]}$ are also lattice path matroid polytopes.
\end{lemma}

\begin{proof}
It is clear that it is sufficient to prove the claim for the facets of an LPM polytope $P_{M[P,Q]}$. We utilize the characterization of facets of matroids polytopes described in \cite[Proposition 7]{joswig2017matroids} which says that facets of a matroid polytope are either induced by hypersimplex facets or hypersimplex splits. If the facet is induced by a hypersimplex facet, we know that these correspond to matroidal deletions and contractions, and LPM are closed under these operations \cite{bonin2006lattice}. Hence, the facet of an LPM polytope is again an LPM polytope in this case. Alternatively, if the facet is induced by a hypersimplex split, we know that it is induced by a \emph{F-hyperplane} \cite{joswig2017matroids} where $F$ is a flat of the LPM $M[P,Q]$, such that $0 < \text{rank}(F) < \#F$, in which case the facet can be described as $P_{M[P,Q]}(F) = P_{(M[P,Q]  \> | \> F \oplus M[P,Q]/ F)}$. Since LPM are also closed under direct sum and restrictions \cite{bonin2010lattice}, therefore this facet is again a LPM polytope.     
\end{proof}

We do highlight the fact that a characterization of snake polytopes does exist, and it concludes that snake polytopes are unimodular equivalent to order polytopes of zig-zag posets \cite[Theorem 4.7]{Knauer2018}. Additionally, the facial structure of LPM has also been classified in terms of certain sets of deletions, contractions and direct sums in \cite{an2020facial}.

\begin{remark}
Lemma \ref{lem:LPM_faces_are_lpm} appears also as a result in \cite[Theorem 3.12]{bidkhori2012lattice} however the argument there appears incomplete since they only consider hypersimplex facets in their proof and not the facets that get induced via hyperplane splits.   
\end{remark}

Another important connection to our results which we want to highlight is the work of Fink and Rincon on \emph{Stiefel tropical linear spaces} \cite{fink2015stiefel}. For the uninitiated, the \emph{Stiefel} map assigns a $k \times n$ matrix over a field $\mathbb{K}$ to an element in the Grassmannian $Gr(k,n)$. The authors study the tropicalization of this map and also study the properties of its image, called the \emph{Stiefel} image, inside the tropical Grassmannian. The authors in \cite{fink2015stiefel} relate the points inside the Stiefel image, to the class of \emph{regular transversal matroid subdivisions}, which as the name suggests is the class of regular matroidal subdivisions where each maximal cell corresponds to a transversal matroid. Since LPM are transversal, LPM subdivisions are also transversal matroid subdivisions. Additionally, we obtain the following corollary as a direct consequence of \cite[Theorem 6.20]{fink2022presentations}

\begin{corollary}
Let $L$ be the tropical linear space dual to a LPM subdivision. Then $L$ lies in the corresponding Stiefel image.  
\end{corollary}

In \cite[Proposition 5.1]{fink2015stiefel} a facet description for transversal matroid polytopes is provided and \cite[Corollary 6.21]{fink2022presentations} provides a partial characterization of transversal matroids in terms of its facets. Based on these results, we propose the following question

\begin{question}\label{ques:Question1}
Let $P_{M}$ be the matroid polytope of a matroid $M$, such that all of its faces are LPM polytopes. Does this imply that $M$ is also a LPM polytope?     
\end{question}

We observe that an affirmative answer, along with Lemma \ref{lem:LPM_faces_are_lpm} would provide a full characterization of LPM polytopes in terms of their faces. Also, we already know due to prior results that with the assumptions in the question, $M$ is both transversal \cite[Corollary 6.20]{fink2022presentations} and a positroid \cite[Theorem 3.20]{m=2amplut}. Hence, it is also worthwhile to inquire about the ways in which the three different classes of matroids namely; transversal matroids, lattice path matroids and positroids interact. A subsequent study of relations between Stiefel tropical linear spaces and LPM subdivisions will be explored elsewhere.

We recall that a matroidal subdivision is completely determined by its 3-skeleton \cite{olarte2019local}. In recent work \cite{joswig23}, the authors introduce the class of \emph{permutahedral subdivisions} which is the class of polyhedral subdivisions of generalized permutahedra into cells that are generalized permutahedra. They also that the 2-skeleton of a permutahedral subdivision does not completely determine the subdivision. In the background of these results, we would be happy to understand how the class of \emph{LPM subdivisions} which we introduced in this paper, behave and possibly find out the criterion which completely determines a LPM subdivision.

We also comment on the location of the positroid cells corresponding to LPM in the stratification of the positive Grassmannian. We consider two well-known families of cells in the positive Grassmannian \cite{benedetti2023lattice}

\begin{definition}
A positroid cell $\Pi$ is called a \emph{Schubert cell} if a generic point $U \in \Pi$ gives rise to a representable matroid $\mathcal{M}_{I} = ([n], \mathcal{B})$ where $B \in \mathcal{B}$ if and only if $I <_{1} B$, where $<_{1}$ is the usual total order on $[n]$.     
\end{definition}

\begin{definition}
A positroid cell $\Pi$ is called a \emph{Richardson cell} if a generic point $U \in \Pi$ gives rise to a representable matroid $\mathcal{M}_{I}^{J} = ([n], \mathcal{B})$ where $B \in \mathcal{B}$ if and only if $I <_{1} B <_{1} J$, where $<_{1}$ is the usual total order on $[n]$.   
\end{definition}

Schubert matroids correspond to Schubert cells and lattice path matroids correspond to Richardson cells. We wish to understand these Richardson cells in depth, given the context of lattice path matroids and in the light of questions from algebraic geometry concerning positroid and Richardson varieties as mentioned in \cite{knutson2013positroid,benedetti2023lattice}. We are currently working on a sequel to our work here, in the context of the new definition of lattice path flag matroids \cite{benedetti2023lattice} and to look at equivalent questions in the realm of flag matroids, along with the flag matroid equivalent of the Dressian, i.e, \emph{Flag Dressian} \cite{brandt2021tropical} and the associated tropical flag variety in \cite{tewarifut}.

For our results about the amplituhedron, our results have two facets. Firstly, we provide a matroidal treatment to the well-known BCFW style recurrence relations for positroidal dissections 
of the hypersimplex. For the $m=2$ amplituhedron, via \emph{T-duality} described in \cite{m=2amplut}, these dissections correspond to a dissection of the amplituhedron in terms of \emph{Grasstopes} \cite{m=2amplut}. However, not much is known about the relations between triangulations of the amplituhedron and dissections of the hypersimplex in the case of the $m=4$ amplituhedron. We provide a first counterpart of positroid dissections of the hypersimplex for BCFW triangulations of $\mathcal{A}_{n,k,4}$. We wish to explore the possibility of equivalent notions of T-duality for the $m=4$ amplituhedron as well. We also wish to examine connections between combinatorial objects and LPM's other than the ones discussed here for example \emph{chord diagrams} and \emph{domino bases} described in \cite{even2021amplituhedron}. We also point the reader to recent work done on weakly separated collections and matroidal subdivisions \cite{early2019weakly}, which also correlates to some of your observations and is an interesting avenue for further exploration.

\bibliographystyle{siam}
\bibliography{biblio.bib}

\end{document}